\documentclass[1 [leqno, 11pt]{amsart}
\usepackage{amssymb,  amsmath, amsmath, latexsym, amssymb, amsfonts, amsbsy, amsthm,mathtools, graphicx, CJKutf8, CJKnumb, CJKulem, extarrows, color, dsfont}
\usepackage{amssymb, amsmath,amsmath,latexsym,amssymb,amsfonts,amsbsy, amsthm,mathtools,graphicx,CJKutf8,CJKnumb,CJKulem,color}
\usepackage{makecell}

\usepackage{float}
\usepackage{multirow}
\usepackage{tikz}
\usepackage{subfigure}
\usepackage{hyperref}
\usetikzlibrary{shapes.geometric, arrows,plotmarks,backgrounds}
\usepackage{pgfplots}
\usepackage{upgreek}
\numberwithin{equation}{section}

\allowdisplaybreaks

\let\al=\alpha
\let\b=\beta

\let\d=\delta

\let\s=\sigma
\let\f=\frac

\let\om=\omega
\let\G= \Gamma

\let\ep=\epsilon
\let\Om=\Omega
\let\wt=\widetilde
\let\wh=\widehat

\let\pa=\partial
\let\va=\varphi

\def\cA{\mathcal{A}}

\def\cD{\mathcal{D}}

\def\cH{\mathcal{H}}
\def\cI{\mathcal{I}}

\def\cS{\mathcal{S}}

\def\no{\noindent}

\def\bbT{\mathbb{T}}

\newcommand{\beq}{\begin{equation}}
\newcommand{\eeq}{\end{equation}}
\newcommand{\beqno}{\begin{equation*}}
\newcommand{\eeqno}{\end{equation*}}
\newcommand{\ben}{\begin{eqnarray}}
\newcommand{\een}{\end{eqnarray}}
\newcommand{\beno}{\begin{eqnarray*}}
\newcommand{\eeno}{\end{eqnarray*}}

\newtheorem{theorem}{Theorem}[section]
\newtheorem{definition}[theorem]{Definition}
\newtheorem{lemma}[theorem]{Lemma}
\newtheorem{proposition}[theorem]{Proposition}
\newtheorem{corol}[theorem]{Corollary}
\newtheorem{remark}[theorem]{Remark}

\begin{document}
\begin{CJK*}{GBK}{song}

\title[]{Linear damping and depletion in flowing plasma with strong sheared magnetic fields}

\author{Han Liu}
\address{Department of Mathematics, New York University Abu Dhabi, Saadiyat Island, P.O. Box 129188, Abu Dhabi, United Arab Emirates.}
\email{hl4294@nyu.edu}
\author{Nader Masmoudi}
\address{Department of Mathematics, New York University Abu Dhabi, Saadiyat Island, P.O. Box 129188, Abu Dhabi, United Arab Emirates; Courant Institute of Mathematical Sciences, New York University, 251 Mercer Street, New York, NY 10012, USA.}
\email{masmoudi@cims.nyu.edu}
\author{Cuili Zhai}
\address{School of Mathematics and Physics, University of Science and Technology Beijing, 100083, Beijing, P. R. China.}
\email{zhaicuili035@126.com, cuilizhai@ustb.edu.cn}
\author{Weiren Zhao}
\address{Department of Mathematics, New York University Abu Dhabi, Saadiyat Island, P.O. Box 129188, Abu Dhabi, United Arab Emirates.}
\email{zjzjzwr@126.com, wz19@nyu.edu}

\maketitle

\begin{abstract}
In this paper, we study the long-time behavior of the solution for the linearized ideal MHD around sheared velocity and magnetic field under Stern stability condition. We prove that the velocity and magnetic field will converge to sheared velocity and magnetic field as time approaches infinity. Moreover a new depletion phenomenon is proved: the horizontal velocity and magnetic field at the critical points will decay to $0$ as time approaches infinity.
\end{abstract}

\section{Introduction}
The appearance of large coherent structures is an important phenomena in the magnetic fluid. The study of the long-time behavior of MHD waves is a very active field in physics and mathematics \cite{Heyvaerts1982,Lifschitz1989,Zohm}. 

\subsection{The linearized MHD system}

In this paper, we consider the magnetic fluid described by the two-dimensional incompressible ideal MHD equations on the periodic domain $\mathbb{T}^2=\{(x,y)\vert x\in\mathbb{T}, y\in\mathbb{T}\}$
\begin{equation}\label{ideal MHD}
 \left\{\begin{array}{l}
\partial_t \mathcal U+ \mathcal U\cdot\nabla \mathcal U- \mathcal H\cdot\nabla \mathcal H+\nabla \mathcal P=0,\\
\partial_t \mathcal H+ \mathcal U\cdot\nabla \mathcal H- \mathcal H\cdot\nabla \mathcal U=0,\\
\nabla\cdot \mathcal U=0,\ \ \nabla\cdot \mathcal H=0,
\end{array}\right.
\end{equation}
with initial data $\mathcal U(0,x,y)$ and $\mathcal H(0,x,y)$. 
Here $\mathcal U=(\mathcal U_1, \mathcal U_2)$, $\mathcal H=(\mathcal H_1, \mathcal H_2)$ and $\mathcal P$ denote the velocity field, magnetic field, and the total pressure of the magnetic fluid, respectively.

System \eqref{ideal MHD} has an equilibrium $\mathcal U_s=(u(y),0)$, $\mathcal H_s=(b(y),0)$, $\mathcal P_s= \text{const.}$. We shall focus on the asymptotic behavior of the linearized 2D MHD equations around this equilibrium, which take the form
 \begin{equation}\label{linearized MHD}
\left\{\begin{array}{l}
\partial_tU_1+u\partial_xU_1+\partial_xp+u'U_2-b\partial_xH_1-b'H_2=0,\\
\partial_tU_2+u\partial_xU_2+\partial_yp-b\partial_xH_2=0,\\
\partial_tH_1+u\partial_xH_1+b'U_2-b\partial_xU_1-u'H_2=0,\\
\partial_tH_2+u\partial_xH_2-b\partial_xU_2=0,\\
\nabla\cdot U=0,\ \ \nabla\cdot H=0.
\end{array}\right.
\end{equation}

We introduce the vorticity $\omega =\partial_x U_2-\partial_y U_1$ and the current density $j=\partial_x H_2-\partial_y H_1$ which satisfy the following equations:
\begin{equation}\label{vorticity and current density}
 \left\{\begin{array}{l}
\partial_t\om+u\partial_x\om-b\partial_xj=u''U_2-b''H_2,\\
\partial_tj+u\partial_xj-b\partial_x\om=b''U_2-u''H_2+u'\partial_xH_1
-u'\partial_yH_2+b'\partial_yU_2-b'\partial_xU_1.
\end{array}\right.
\end{equation}
We further introduce the stream function $\psi$ and the magnetic potential function $ \phi,$ satisfying $ U=(\partial_y \psi, -\partial_x \psi)$, $\om=-\Delta \psi$ and  $H=(\partial_y \phi, -\partial_x \phi)$, $j=-\Delta \phi,$ which allow us to derive the following system, satisfied by $( \psi,  \phi)$:
\begin{equation}\label{psi phi}
 \left\{\begin{array}{l}
\partial_t(\Delta\psi)+u\partial_x(\Delta\psi)-b\partial_x(\Delta\phi)=u''\partial_x\psi
-b''\partial_x\phi,\\
\partial_t(\Delta\phi)+u\partial_x(\Delta\phi)-b\partial_x(\Delta\psi)=b''\partial_x\psi
-u''\partial_x\phi-2u'\partial_x\partial_y\phi
+2b'\partial_x\partial_y\psi.
\end{array}\right.
\end{equation}
Taking the Fourier transform in $x$ and inverting the operator $(\partial_y^2-\alpha^2),$ we rewrite the system as 
\begin{equation}\label{Fourier equations2}
\left\{\begin{aligned}
&\partial_t\Big(
  \begin{array}{ccc}
    \widehat{\psi}\\
    \widehat{\phi}\\
  \end{array}
\Big)(t,\al,y)=-i\alpha M_{\al}\Big(
  \begin{array}{ccc}
     \widehat{\psi}\\
    \widehat{\phi}\\
  \end{array}
\Big)(t,\al,y),\\
&(\widehat{U}_1,\widehat{U}_2)=(\pa_y\widehat{ \psi},-i\al \widehat{ \psi}), \ (\widehat{H}_1, \widehat{H}_2)=(\pa_y\widehat{ \phi},-i\al \widehat{ \phi}),
\end{aligned}\right.
\end{equation}
where $\alpha\neq0$ and 
\beq\label{eq: matrix}
M_{\al}=-\Delta_{\al}^{-1}
\left[
\begin{matrix}
& u''-u\Delta_{\al} & -b''+b\Delta_{\al}\\
& b\Delta_{\al}+b''+2b'\pa_y & -u\Delta_{\al}-u''-2u'\pa_y
\end{matrix}
\right].
\eeq

For the homogenous equilibrium $u=0, \, b(y)= \text{const.}$, $\widehat {U}_2$ and $\widehat {H}_2$ satisfy a 1-D wave equation, which is stable but exhibits no decay. There are few rigorous mathematical results on the non-flowing plasma with inhomogeneous sheared magnetic field. In \cite{TG1973}, Tataronis and Grossmann predicted that the vertical components of velocity and magnetic field may decay by phase mixing, to which a mathematically rigorous proof was given by Ren and Zhao in \cite{RZ2017}, under the condition that the magnetic field is positive and strictly monotone. If the positivity assumption on the magnetic field is removed, which allows the direction of the sheared magnetic field to change, then it turns out that magnetic reconnection occurs in infinite time, generating the magnetic island. This phenomenon was predicted by Hirota, Tatsuno and Yoshida \cite{HTY05} and later justified by Zhai, Zhang and Zhao \cite{ZZZ2018}. For the flowing plasma $u\neq 0$, fewer mathematical rigorous results are available. We refer to \cite{HTY05,RWZ2020, ZZZ2018} for the long time behaviors of the solutions to the MHD equations linearized around a flowing plasma. 

\no{\bf Notations:} Let us specify the notations to be used throughout the paper. We denote by $A \lesssim B$ an estimate of the form $A \leq C B$ and by $A \sim B$ an estimate of the form $C^{-1} B \leq A \leq C B,$ where $C$ is a constant. Given a function $f(x,y),$ we denote its Fourier transform in $x$-variable as $\wh f(\alpha, y)=\f{1}{2\pi}\int_{\mathbb{T}}e^{-ix\al}f(x,y)dx,$ where $\alpha$ is the wave number. We shall use the Japanese bracket notation $\langle x \rangle:= \sqrt{|x|^2+1}.$

\subsection{Vertical damping and horizontal depletion}
In this paper, we focus on the long time behavior of the solution to the linearized MHD equation \eqref{Fourier equations2}. 

Our first main result states as follows: 
\begin{theorem}[Vertical damping]\label{main thm2}
Let $u, b\in C^3(\mathbb{T})$ be such that $b > |u| \geq 0$ and the critical points of $(u \pm b)$ are non-degenerate. Let $\al\neq 0$ be a fixed wave number and let $\left(\widehat{ \psi},\widehat{ \phi} \right)$ solve \eqref{Fourier equations2} with initial data $\left(\widehat{ \psi}_0,\widehat{ \phi}_0 \right)\in (H^3\times H^3)$. Then the following space-time estimate holds:
\begin{equation}\label{estsptm}
\left \| \left(\widehat{ \psi},\widehat{ \phi} \right) \right\|_{H^1_tL^2_y}\leq C_{\al} \left\| \left(\widehat{ \psi}_0,\widehat{ \phi}_0 \right) \right\|_{H^3_y}. 
\end{equation}
In particular, $\lim\limits_{t\to \infty}\left\|\left(\widehat {U}_2,\, \widehat{H}_2\right)\right\|_{L^2_y}=0$.
\end{theorem}

\begin{remark}
The condition $|u|<|b|$ is called Stern stability condition (see \cite{Stern}).
\end{remark}
\begin{remark}
Formally, the space-time estimate may indicate that $ \left\| \left(\widehat {U}_2,\, \widehat{H}_2 \right) \right \|_{L^2}\lesssim \frac{1}{\langle t \rangle^{\beta}}$ with $\b>{\f12}$. The study of the precise decay rate shall be our forthcoming work. 
\end{remark}
\begin{remark}\label{energy conservation rem}
For the case of flowing plasma with constant velocity or non-flowing plasma $(u=\mathrm{const}.)$, it holds that
\begin{equation}\label{engcsv}
\|U(t),H(t)\|_{H^k_xL^2_y}\sim \|U_0,H_0\|_{H^k_xL^2_y}, \ k\geq 0,
\end{equation}
which implies linear growth of vorticity and current density, i.e., 
\beno
\|\om(t),j(t)\|_{L^2_{x,y}}\lesssim \langle t\rangle \|\om_{0},j_0\|_{L^2_{x,y}}. 
\eeno
The proof can be found in Appendix C.
By \eqref{engcsv}, the total energy is almost conserved. Hence, the vertical damping in Theorem \ref{main thm2} shows that energy is transferred from the vertical direction to the horizontal direction. Whether similar energy conservation results are true for the flowing plasma with $u \not =\mathrm{const.}$ remains an open question.  
\end{remark}


The vertical damping is induced by a certain mixing mechanism similar to the vorticity mixing that leads to inviscid damping for linearized Euler equations (see \cite{jia2020linear_Gevrey,jia2020linear,WZZ18,WZZ1,WZZ2,CZ17}). Readers may consult \cite{Bedrossian_2015,Deng_Masmoudi_2018,ionescu_jia_2019_couette,ionescu2020nonlinear,Lin_2011,masmoudi_zhao2020nonlinear} for recent progress in nonlinear inviscid damping.

To better illustrate the mixing mechanism, let us recall the system in terms of $({U}_1, {H}_1)$:
\begin{equation}\label{linearized MHD U_1, H_1}
\left\{\begin{array}{l}
\partial_t U_1+u\pa_xU_1-b\partial_x H_1=L_1,\\
\partial_t H_1+u\pa_xH_1-b\partial_x U_1=L_2,
\end{array}\right.
\end{equation}
where $(L_1,L_2):= \left(b' H_2-u'U_2-2\pa_x\Delta^{-1}(b'\pa_x  H_2-u'\pa_xU_2), u'H_2-b' U_2 \right)$ can be seen as nonlocal forcing terms depending on ${U}_2$ and ${H}_2.$

By the incompressibility condition, we can check that
\beno
\|(\pa_x {U}_1,\pa_x {H}_1)\|_{L^2_xH^{-1}_y}\sim \|( {U}_2, {H}_2)\|_{L^2_{x,y}}.
\eeno
Then the mixing of $({U}_1\pm {H}_1)$ would lead to the linear damping of $({U}_2,{H}_2).$

Let us consider a toy model, obtained by neglecting the nonlocal forcing terms $(L_1,L_2)$ in the linearized system \eqref{linearized MHD U_1, H_1}, i.e.,
\begin{equation}\label{eq:toy model}
 \left\{\begin{array}{l}
\partial_t U_1+u\pa_xU_1-b\partial_x H_1=0,\\
\partial_t H_1+u\pa_xH_1-b\partial_x U_1=0.
\end{array}\right.
\end{equation}
It is then easy to see that the Els\"asser variables $\mathcal{Z}_1^{\pm} :=U_1\pm H_1$ satisfy certain transport equations and then 
\ben\label{eq:initial}
\begin{split}
\left (\wh{U}_1+\wh{H}_1\right)(t,\al,y )=\wh{\mathcal{Z}}^{+}_{1, in}(\al,y)e^{-i\al (u-b)t},\\
\left (\wh{U}_1-\wh{H}_1\right)(t,\al,y )=\wh{\mathcal{Z}}^{-}_{1, in}(\al,y)e^{-i\al (u+b)t},
\end{split}
\een
with $\wh{\mathcal{Z}}^\pm_{1, in}$ denoting the initial data.

Regarding the toy model \eqref{eq:toy model}, we have the following conclusions. 
\begin{lemma}\label{lem:toy model}
Let $u, b \in C^{3}(\mathbb{T})$ be such that $u\pm b$ have only non-degenerate critical points. Then the solution of \eqref{eq:toy model} with initial data $({U}_{1,in},{H}_{1,in})$ satisfies 
\begin{equation}\label{1/2}
\|(\pa_x{U}_1,\pa_x{H}_1)\|_{L^2_xH^{-1}_y}\lesssim \f{1}{\langle t \rangle^\frac{1}{2}}\|({U}_{1,in},{H}_{1,in})\|_{H^{\f12}_xH^{1}_y}.
\end{equation}
Moreover, if the initial data $({U}_{1,in},{H}_{1,in})$ vanish at all the critical points of $(u\pm b)$, then it holds that 
\begin{equation}\label{-1}
\|(\pa_x{U}_1,\pa_x{H}_1)\|_{L^2_xH^{-1}_y}\lesssim \f{1}{\langle t \rangle}\|({U}_{1,in},{H}_{1,in})\|_{H^{-1}_xH^{2}_y}.
\end{equation}
\end{lemma}
The proof of \eqref{1/2} can be found in \cite{WZZ1}; via the same dual method one can prove \eqref{-1}. In fact, the decay rate of $t^{-1/2}$ for the toy model \eqref{eq:toy model} is optimal,  as we know, via the classical stationary phase approximation (see Chapter VIII of \cite{Stein-book}), that there exists a class of initial data such that the corresponding solutions satisfy 
\beno
\|(\pa_x{U}_1,\pa_x{H}_1)\|_{L^2_xH^{-1}_y}\sim \f{1}{\langle t\rangle^{\f12}}.
\eeno 

The space-time estimate in Theorem \eqref{main thm2}, however, fails to hold for the toy model \eqref{eq:toy model}. Exploring the mechanisms behind the enhanced damping for the complete system \eqref{linearized MHD U_1, H_1}, we found a new dynamical phenomenon apart from velocity mixing: the depletion of horizontal velocity and magnetic field $( {U}_1, {H}_1)$ at the critical points of $u\pm b$. This leads to our second main result, which states as follows:

\begin{theorem}[Horizontal depletion]\label{main thm3}
Let $u, b$ satisfy the same assumptions as in Theorem \ref{main thm2}. 
Let $y_0$ be a critical point of $(u+ b)$ or $(u-b)$, i.e., $u'(y_0)=b'(y_0)$ or $u'(y_0)=-b'(y_0)$. Let $\left(\widehat {U}_1, \widehat{H}_1 \right)$ correspond to the solution to \eqref{Fourier equations2} with initial data $\left(\widehat{ \psi}_0,\widehat{ \phi}_0 \right)\in (H^3\times H^3)$. Then it holds that
$$\lim\limits_{t\to \infty}\left|\left(\widehat {U}_1,\, \widehat{H}_1\right)(t,\al,y_0)\right|=0. $$
\end{theorem}
\begin{remark} 
From \eqref{eq:initial}, we can see that the horizontal depletion in Theorem \ref{main thm3} is not true for the toy model \eqref{eq:toy model}.

For a more precise description of the long time behavior of the horizontal components, we conjecture that there exist some final states ${\mathcal Z}_{1,\infty}^\pm$ and some $\kappa >0$ such that
\ben\label{eq:infinite}
\begin{split}
\left(\wh{U}_1+\wh{H}_1 \right)(t,\al,y)\sim \left(\wh{ \mathcal Z}^+_{1,\infty} \right)(\al,y)e^{-i\al (u-b)t}+\mathcal O(t^{-\kappa}),\\
\left(\wh{U}_1-\wh{H}_1 \right)(t,\al,y)\sim \left(\wh{\mathcal Z}^-_{1, \infty} \right)(\al,y)e^{-i\al (u+b)t}+\mathcal O(t^{-\kappa}).
\end{split}
\een
Theorem \ref{main thm3} reveals that $\wh{\mathcal Z}_{1,\infty}^\pm$ vanish at all the critical points of $(u\pm b).$ Our conjecture \eqref{eq:infinite} shall imply the space-time estimate \eqref{estsptm} and even a precise decay rate, provided that suitable regularity of the final states could be proven.
\end{remark}

The parallel phenomenon exists in the realm of hydrodynamics. Vorticity depletion for the linearized 2D Euler equations around shear flows, first predicted by Bouchet and Morita in \cite{bouchet2010large}, was later mathematically proven in \cite{WZZ1} by Wei, Zhang and Zhao. A similar vorticity depletion result for the 2D Euler equation linearized around a radially symmetric and strictly decreasing vorticity distribution is due to Bedrossian, Coti-Zelati and Vicol \cite{BCV2017}. {\bf As far as we know, this is the first paper studying the depletion of the horizontal velocity and magnetic field for the linearized MHD equations.}

Comparing the toy model \eqref{eq:toy model} with the complete system \eqref{linearized MHD U_1, H_1}, we observe at least three significant effects of the nonlocal terms $L_1$ and $L_2$:
\begin{enumerate}
\item Altering the final state, as $\wh{\mathcal Z}^\pm_{1, in}$ on the right hand side of \eqref{eq:initial} are changed into some other final state $\wh{\mathcal Z}^\pm_{1, \infty}$ on the right hand side of \eqref{eq:infinite}; 
\item Causing $\wh{\mathcal Z}^\pm_{1, \infty}$ to vanish at the critical points of $(u \pm b)$; 
\item Enhancing the damping.
\end{enumerate}This demonstrates that the nonlocal terms $L_1$ and $L_2$ in \eqref{linearized MHD U_1, H_1} cannot simply be neglected or be regarded as mere perturbations on the toy model.

\begin{remark}
The significance of the nonlocal terms can also be seen from the resolvent estimate. Indeed, if we consider the toy model \eqref{eq:toy model} instead, the Sturmian equation \eqref{eq:ODE1} in Section 3 would become $\left((u-c)^2-b^2 \right)\pa_y\Phi=F,$ which is much simpler and yields an obvious estimate 
$$|\pa_y\Phi|\lesssim {|(u-c)^2-b^2|}^{-1}.$$ 
Yet, as we shall see in Section 3, at any critical point $y_0$ of $(u+b)$ or $(u-b),$ the solution $\Phi$ to the actual Sturmian equation \eqref{eq:ODE1} enjoys a non-trivial estimate
$$|\pa_y\Phi(y_0)|\lesssim |(u(y_0)-c)^2-b(y_0)^2|^{-\f34}.$$
This $\frac{1}{4}$-improvement, resulting from the effects of $L_1$ and $L_2,$ is the key to the depletion result. For more details, we refer to Section 3 and 4.
\end{remark}

We have the following comments on the results in Theorem \ref{main thm2} and Theorem \ref{main thm3}.

\begin{remark}
The vertical damping and horizontal depletion results also hold for the case of finite channel under slip boundary condition, provided that the critical points do not appear at the boundary. 
\end{remark}

\begin{remark}
To highlight the differences among the long time behaviors of the solutions to the linearized MHD equations in various cases, we show the following table:
\begin{table}[H]
\centering
\resizebox{\textwidth}{!}{
\begin{tabular}{|c|c|c|c|c|c|}
\hline  
\multicolumn{4}{|c|}{Conditions}&\multirow{2}*{Results}&\multirow{2}*{References}\\
\cline{1-4}
Monotonicity&\makecell{Uniform direction \\$b>0$} &\makecell{Stern stability\\ $|u|\leq|b|$}  &Other  conditions&{}&\\
\hline 
Yes&Yes&Yes&$u\equiv 0$&Damping&\cite{RZ2017}\\
\hline
Yes&No&Yes&$u(0)=b(0)=0$&Magnetic  Island&\cite{ZZZ2018}\\
\hline
Yes&No&No&\makecell{$u=k_1y,\, b=k_2y$\\ $k_1>k_2\geq 0$}&Damping&\cite{RWZ2020}\\
\hline
No&Yes&\makecell{Yes\\$|u|<b$}&\makecell[l]{Non-degenerate\\ critical points} &\makecell{Damping \\$\&$ Depletion}& This paper\\
\hline
\end{tabular}
}
\end{table}
Comparing the results listed above, we have the following observations:
\begin{enumerate}
\item The first result of this paper, in comparison with those in \cite{RZ2017} and \cite{ZZZ2018}, indicates that the unidirectionality of the sheared magnetic field, rather than the monotonicity, is responsible for the damping. 
\item The stability condition $|u| <b$ in this paper seems to indicate magnetism-driven mechanisms, whilst the damping in \cite{RWZ2020} might be seen as fluid-driven.
\end{enumerate}
\end{remark}

\subsection{Outline of the paper}
The organization of the paper is as follows. In Section 2, we introduce the representation formula, which is a contour integral of the resolvent. The study of the resolvent is then reduced to that of the Sturmian equation. In Section 3, we study the Sturmian equation and establish a uniform estimate as well as the limiting absorption principle. This is the most technical part of the paper as our assumptions include rather general sheared velocity and magnetic field, which in turn results in a range of situations that need to be discussed separately from each other. A novelty here is the use of ODE techniques in dealing with the situation involving the critical points. In Section 4, we prove the main theorems by the resolvent estimate.

\section{The Dunford integral and the Sturmian equation}
The basic idea for the study of the long-time behavior of the solution to \eqref{Fourier equations2} is to acquire a precise formula of $(\psi,\phi),$ which requires understanding the spectral properties of the linearized operator $M_{\al}$. Indeed, it is easy to check that the spectrum $\s(M_{\al})=\mathrm{Ran}\, (u+b)\cup \mathrm{Ran}\, (u-b)$. Then we have the Dunford integral
\begin{equation}\label{eq: psi and phi-1}
\begin{split}
\Big(
  \begin{array}{ccc}
    \widehat{\psi}\\
    \widehat{\phi}\\
  \end{array}
\Big)(t,\alpha,y)
&=\frac{1}{2\pi i}\int_{\partial\Omega}e^{-i\alpha tc}(cI-M_{\al})^{-1}\Big(
  \begin{array}{ccc}
    \widehat{\psi}\\
    \widehat{\phi}\\
  \end{array}
\Big)(0,\alpha,y)dc\\
&=\lim_{\ep\to 0+}\frac{1}{2\pi i}\int_{\partial\Omega_{\ep}}e^{-i\alpha tc}(cI-M_{\al})^{-1}\Big(
  \begin{array}{ccc}
    \widehat{\psi}\\
    \widehat{\phi}\\
  \end{array}
\Big)(0,\alpha,y)dc,
\end{split}
\end{equation}
where $\Om$ contains the spectrum $\s(M_{\al})$ and $\Om_{\ep}$ is the $\ep$-neighborhood of $\s(M_{\al})$. Here the last equality holds due to the fact that for $c\notin \s(M_{\al})$, the resolvent $(cI-M_{\al})^{-1}\Big(
  \begin{array}{ccc}
    \widehat{\psi}\\
    \widehat{\phi}\\
  \end{array}
\Big)(0,\alpha,y)$ is analytic in $c$. 
With the representation formula \eqref{eq: psi and phi-1}, we have reduced our problem to the study of the resolvent $(cI-M_{\al})^{-1}$ and the limit of the contour integral.

Assume that 
$$\big(cI-M_{\al}\big)^{-1}\Big(\begin{array}{l} \widehat{\psi}_0\\ \widehat{\phi}_0\end{array}\Big)(\al,y)=\Big(\begin{array}{l} \Psi_1\\ \Phi_1\end{array}\Big)(\al,y,c).$$ 
Then a direct calculation shows that $(\Psi_1,\Phi_1)$ solves the following system, for $c\notin \s(M_{\al})$:
\begin{align*}
\left\{
\begin{array}{l}
(u-c)\Delta_{\al}\Psi_1-u''\Psi_1-b\Delta_{\al}\Phi_1
+b''\Phi_1
=\Delta_{\al}\widehat{\psi}_0,\\
(u-c)\Phi_1-b\Psi_1=-\widehat{\phi}_0.
\end{array}
\right.
\end{align*}
Here $\Delta_{\al}=\partial_y^2-\al^2$. Let $\Phi_1(\al,y,c)=b(y)\Phi(\al,y,c)$, and then 
\beno
\Psi_1(\al,y,c)=(u(y)-c)\Phi(\al,y,c)+\widehat{\phi}_0(\al,y)/b(y).
\eeno
 Thus, we obtain
\beq\label{eq:ODE1}
\begin{split}
&\pa_y\left(\left(\left(u(y)-c\right)^2-b(y)^2\right)\pa_y\Phi(\al, y,c)\right)
-\al^2\left(\left(u(y)-c\right)^2-b(y)^2\right)\Phi(\al,y,c)\\
&
=\Delta_{\al}\widehat{\psi}_0(\al,y)-\big(u(y)-c\big)\Delta_{\al}\left(\frac{\widehat{\phi}_0(\al,y)}{b(y)}\right)
+u''(y)\frac{\widehat{\phi}_0(\al,y)}{b(y)} := F(\al,y,c).
\end{split}
\eeq
This is the so-called Sturmian type equation.

\section{Uniform estimates and limiting absorption principle}\label{sec: lap}

Let $\Omega_{\ep_0}$ denote the $\ep_0$-neighborhood of $\text{Ran}\,(u+b) \, \cup \,\text{Ran}\, (u-b)$ in $\mathbb C.$ We shall study Equation \eqref{eq:ODE1} with $c \in \Omega_{\ep_0}.$ We introduce the Els\"asser variables $Z_{\pm}:=u\pm b.$ Hence, we can rewrite the equation of $\Phi(\al, y, c)$ as
\beq\label{eq:odeihom}
\partial_y((Z_--c)(Z_+-c)\partial_y \Phi) - \alpha^2 (Z_--c)(Z_+-c)\Phi =F.
\eeq

This section is devoted to the proofs of the following propositions, which are crucial to our main results.

\begin{proposition}\label{prop: lap} There exists $\ep_0 >0$ such that for $c \in \left( \Omega_{\ep_0} \setminus \left(\mathrm{Ran}\, Z_+ \cup \mathrm{Ran}\, Z_-\right) \right),$ the solution to \eqref{eq:odeihom} satisfies the following bound, uniform with respect to $c$
$$ \|\Phi (\al,\cdot, c)\|_{L^2} + \| (Z_--c)(Z_+-c)\partial_y \Phi (\al,\cdot, c) \|_{H^1} \leq C \|F(\al,\cdot, c)\|_{H^1}.$$
\end{proposition}
The estimate on $\Phi$ can be continued (in $c$) up to the boundary $\mathrm{Ran}\, Z_+ \cup \mathrm{Ran}\, Z_-.$ 
\begin{proposition}\label{prop: conti}
For $c \in \left(\mathrm{Ran}\, Z_+ \cup \mathrm{Ran}\, Z_-\right),$ there exist $\Phi^\pm(\al,\cdot,c) \in L^2$ such that as $\ep \to 0^+,$ $\Phi(\al,\cdot,c\pm i \ep) \to \Phi^\pm(\al,\cdot,c)$ in $L^r $ with $r \in (1,2)$ and 
$$ \left\|\Phi^\pm(\al,\cdot, c) \right\|_{L^2}  \leq C\|F(\al,\cdot, c)\|_{H^1}.$$
\end{proposition}

We shall prove Proposition \ref{prop: lap} by contradiction. Suppose that the proposition is false,  then there exists a sequence $\{ c_n, \Phi_n,  F_n\}_{n=1}^\infty$ satisfying the equation
\begin{equation}\label{eq:eln00}
\partial_y ((Z_--c_n)(Z_+-c_n) \partial_y \Phi_n ) -\alpha^2 (Z_--c_n)(Z_+-c_n) \Phi_n = F_n,
\end{equation}
such that 
\begin{itemize}
\item $c_n \in \left( \Omega_{\ep_0} \setminus \big(\mathrm{Ran}\, Z_+\cup \mathrm{Ran}\, Z_-\big) \right),$
\item $\| \Phi_n\|_{L^2}+\|(Z_--c_n)(Z_+-c_n)\partial_y \Phi_n\|_{H^1} =1$,
\item $F_n \to 0$ in $H^1$ as $n \to \infty.$
\end{itemize}
For convenience, we shall use the notation $q_n:= (Z_--c_n)(Z_+-c_n)\pa_y \Phi_n$ from time to time in this section.

As $| c_n | \leq C$ and $\| \Phi_n\|_{L^2}+\|q_n\|_{H^1} =1,$ we know that 
\begin{itemize}
\item $c_n \to c$ for some $c \in \overline{\Omega}_{\ep_0}$ up to a subsequence,
\item $\Phi_n \rightharpoonup \Phi$ in $L^2$ up to a subsequence,
\item $q_n \rightharpoonup q$ in $H^1$ up to a subsequence, for some $q \in H^1.$
\end{itemize}
(For convenience, we shall simply use the original $c_n$ and $\Phi_n$ to denote the elements in the subsequence.) Moreover, $q = (Z_--c)(Z_+-c) \pa_y \Phi,$ as seen from the identity
\begin{align*}
\langle q_n, f\rangle_{L^2} = & \langle (Z_--c_n)(Z_+-c_n) \pa_y \Phi_n, f \rangle_{L^2}\\
= & - \langle \Phi_n, Z'_-(Z_+-c_n) f+(Z_--c_n)Z'_+ f +(Z_--c_n)(Z_+-c_n) f' \rangle_{L^2}, \ \forall\,  f\in C^\infty_0.
\end{align*}


We shall show in the following passages that 
\begin{itemize}
\item the weak limit $\Phi \equiv 0,$ 
\item in fact, strong convergences hold true, i.e., $\Phi_n \rightarrow 0$ in $L^2$ and $q_n \rightarrow 0$ in $H^1,$ 
\end{itemize}
which contradict the very assumption that $\| \Phi_n\|_{L^2} + \|(Z_--c_n)(Z_+-c_n)  \pa_y \Phi_n \|_{H^1} =1.$ 

Formally, by performing integration by parts on the limiting equation 
\begin{equation}\label{eq:limiting eq}
\partial_y ((Z_--c)(Z_+-c) \partial_y \Phi ) -\alpha^2 (Z_--c)(Z_+-c) \Phi =0,
\end{equation}
we can see that the weak limit $\Phi \equiv 0.$ To prove this, we have to show that $\Phi$ belongs to a space for which the operation is allowed. 

The limit $c$ belongs to either $\left( \overline{\Omega}_{\ep_0} \setminus \left( \text{Ran}\,Z_+ \cup \text{Ran}\, Z_- \right) \right)$ or $\left( \text{Ran}\,Z_+ \cup \text{Ran}\, Z_- \right).$

Let us first consider the case when $c_n \to c=\text{Re}\, c+i \text{Im}\,c \notin\left( \mathrm{Ran}\, Z_+ \cup \mathrm{Ran}\,Z_- \right).$ In this case, it's straightforward to show that $\Phi_n \to \Phi$ in $L^2$ and $q_n\to q:=(Z_--c)(Z_+-c)  \pa_y \Phi$ in $H^1,$ where $\Phi$ is the classical solution to \eqref{eq:limiting eq}. Taking the inner product with $\overline{\Phi},$ integrating by parts and separating the real and imaginary parts, we obtain
\beq\label{re}
\int_{\mathbb T} \left( \left(u(y)-\mathrm{Re}\,c\right)^2 - \left( b(y)\right)^2-(\mathrm{Im}\,c) ^2\right)
\left(|\pa_y\Phi(y,c)|^2+\al^2|\Phi(y,c)|^2\right)\mathrm{d}y=0,
\eeq
and 
\ben\label{im}
\int_{\mathbb T}(u(y)-\mathrm{Re}\,c)\left(|\pa_y\Phi(y,c)|^2+\al^2|\Phi(y,c)|^2\right) \mathrm{d}y=0.
\een
Multiplying (\ref{im}) by $2\mathrm{Re}\,c$ and adding it to (\ref{re}) lead to
\beno
\int_{\mathbb T} \left( u(y)^2-b(y)^2-(\mathrm{Re}\,c)^2-(\mathrm{Im}\,c)^2\right)
\left(|\pa_y\Phi(y,c)|^2+\al^2|\Phi(y,c)|^2\right)\mathrm{d}y=0.
\eeno
The condition $|u|<b$ then guarantees that $\Phi \equiv 0.$

The difficult situation is when $ c_n \to c \in \left( \text{Ran}\, Z_+ \cup \text{Ran}\, Z_- \right)$ as $n \to \infty,$ which renders Equation \eqref{eq:eln00} degenerate at the the sets $(Z_+)^{-1}(c):= \left\{ y \in \mathbb{T} : Z_+ (y) = c\right\}$ and  $(Z_-)^{-1}(c):= \left\{ y \in \mathbb{T} : Z_- (y) = c\right\}.$  The condition $0< |u| < b$ ensures that $\text{Ran}\, Z_+\cap \text{Ran}\, Z_-=\emptyset.$ Hence, if $c_ n \to c\in \text{Ran}\, Z_+,$ then $c \notin \text{Ran}\, Z_-,$ and vice versa. In this section, we will always provide detailed proofs only for the case that $c_n \to c \in \text{Ran}\, Z_+,$ as those for the case $c_n \to c \in \text{Ran}\, Z_-$ are essentially similar.

By our assumptions on $(u\pm b),$ the sets $(Z_s)^{-1}(c),$ $s=+$ or $-,$ consist of two possible types of points:
\begin{enumerate}
\item  points located in a monotone interval, i.e., $\{y\in\mathbb{T} : Z_s(y)= c \text{ and } |Z_s'(y)| >0\},$ 
\item  critical points, where $Z_s'=0$ and $|Z_s''|>0.$ 
\end{enumerate}
In our analysis,  the two situations require separate treatments.  To better distinguish between the two,  we denote a critical point as $y_0.$ 

\subsection{Weak convergence of $\{\Phi_n\}_{n=1}^\infty$ to $0$}

In this subsection, we establish several lemmas implying that $\Phi_n \rightharpoonup \Phi \equiv 0$ as $c_n\to c$. As previously mentioned, to this end we need to prove that $\Phi$ is regular enough such that the desired integration by parts is justified. This is clear when $y \not \in (Z_{-})^{-1}(c)\cup (Z_+)^{-1}(c),$ as $\Phi\in H^3_{loc}\left(\mathbb{T}\setminus (Z_\pm)^{-1}(c) \right)$ thanks to $|(Z_--c)(Z_+-c)| > C >0.$ Therefore, our consideration starts from the case when $y_c \in (Z_{s})^{-1}(c)$ lies in a region where $Z_{s}$ is strictly monotone, $s=+$ or $-.$ 

\begin{lemma}\label{lem: m} 
Let $\{c_n\}_{n=1}^\infty \subset \left(\Omega_{\ep_0}\setminus \left(\mathrm{Ran}\,Z_+\cup \mathrm{Ran}\, Z_- \right) \right)$ be such that $c_n \to c$ as $n \to \infty$ for a certain $c \in \left(\mathrm{Ran}\, Z_+ \cup \mathrm{Ran}\,Z_- \right).$ Let the triple $\{c_n,  \Phi_n,  F_n  \}_{n=1}^\infty$ satisfy the equation \eqref{eq:eln00} along with the following conditions
\begin{itemize}
\item $\Phi_n \rightharpoonup \Phi$ in $L^2$ and $(Z_--c_n)(Z_+-c_n) \pa_y \Phi_n \rightharpoonup (Z_--c)(Z_+-c)\pa_y \Phi$ in $H^1$ in $\mathcal I,$
\item $F_n \to F$ in $H^1$ in $\mathcal I,$
\end{itemize}  
for some interval $\mathcal{I}:= [y_1,  y_2] \subset \mathbb T$ such that there exists $y_c \in (Z_s)^{-1}(c)$ in $\mathcal I$ with $|Z_s'(y_c)|>0;$ $Z_s'(y_c) Z_s'(y) >0,$ $\forall\, y \in \mathcal{I},$ where $s=+$ or $-.$

Then $\Phi_n \rightarrow \Phi$ in $L^2(\mathcal I)$ and $(Z_--c_n)(Z_+-c_n) \pa_y \Phi_n \rightarrow (Z_--c)(Z_+-c)\pa_y \Phi$ in $H^1(\mathcal I)$. 

In particular, if $F \equiv 0,$ then $\Phi\in H^1(\mathcal{I})$ and
\begin{align}\label{eq:ibpmn}
-\int_{\mathcal{I}} (Z_--c)(Z_+-c)\pa_y\Phi f'\, \mathrm{d}y+\al^2\int_{\mathcal{I}} (Z_--c)(Z_+-c)\Phi f \,\mathrm{d}y= 0 , \ \forall  f \in H^1_{w,0}(\mathcal{I}),
\end{align}
where $H^1_{w,0} =\left\{ f \in L^2 : \left((y-y_c)\pa_y f \right) \in L^2 \text{ and } f\vert_{y=y_1} =f\vert_{y=y_2} =0 \right\}.$
\end{lemma}

\begin{proof} 
To facilitate the proof, let us assume that $c\in \mathrm{Ran}\, Z_+$. Then for $\ep<\f{\min_{y \in \mathbb T} b-\|u\|_{L^{\infty}}}{3}$, there exists $N>0$ such that for any $n>N$, $|c_n-c|<\ep.$ Thus there exists $C>0$ such that $|Z_--c_n|>C^{-1}>0$ for $n>N$. Without loss of generality, let us also assume that $n>N$ and $\mathrm{Im}\, c_n>0$. Thus, there exists $y_{c_n}\in \mathcal I$ such that $Z_+(y_{c_n})=\mathrm{Re}\, c_n$ and $y_{c_n}\to y_c$ as $c_n \to c.$

We recall $q_n= (Z_--c_n)(Z_+-c_n)\partial_y \Phi_n$ with $c_n \in \Omega_{\ep_0}\setminus \left(\mathrm{Ran}\, Z_+\cup \mathrm{Ran}\, Z_-\right).$ Dividing \eqref{eq:eln00} by $(Z_- -c_n)(Z_+ - c_n)$ and differentiating, we see that $q_n$ solves the following equation
\begin{equation}\label{eq:odeq}
\partial_y \left( \frac{\partial_y q_n }{(Z_--c_n)(Z_+-c_n)} \right) - \alpha^2 \left( \frac{q_n}{(Z_--c_n)(Z_+-c_n)} \right) = \partial_y \left( \frac{F_n}{(Z_--c_n)(Z_+-c_n)} \right).
\end{equation}
By Equation \eqref{eq:eln00},  $q_n$ also satisfies  
\begin{eqnarray}
\label{eq:eqn1} & q'_n  = F_n + \alpha^2 (Z_--c_n)(Z_+-c_n) \Phi_n,\\
\label{eq:eqn2} & q''_n  -\alpha^2 q_n = F'_n + \alpha^2 Z'_-(Z_+-c_n) \Phi_n+\alpha^2 (Z_--c_n)Z'_+ \Phi_n. 
\end{eqnarray}
From Equation \eqref{eq:eqn2} and the assumptions $\Phi_n \rightharpoonup \Phi$ in $L^2,$ $q_n \rightharpoonup q:= (Z_--c)(Z_+-c)\pa_y \Phi$ in $H^1$ and $F_n \to F$ in $H^1,$ we can infer that 
\begin{equation}\label{es2}
\| q''_n \|_{L^2} \leq \| F'_n\|_{L^2} + \alpha^2 \left( \|\Phi_n\|_{L^2} + \|q_n\|_{H^1} \right),
\end{equation}
which, along with the assumption that $q_n \rightharpoonup(Z_--c)(Z_+-c)\pa_y \Phi$ in $H^1,$ implies strong convergence, i.e.,
$$(Z_--c_n)(Z_+-c_n) \pa_y \Phi_n \rightarrow (Z_--c)(Z_+-c)\pa_y \Phi \text{ in }H^1.$$

By our assumption and Sobolev embedding, $\|q_n\|_{L^\infty(\mathcal I)} \leq \|q_n \|_{H^1(\mathcal I)} < C,$ it is true that $$\|(y-y_{c_n})\pa_y\Phi_n \|_{L^p(\mathcal I)} < C, \ \forall \, p<\infty.$$

Applying the compactness result from Appendix \ref{rmk1}, we have 
$$\Phi_n \to \Phi \text{ in }L^2.$$

Now we consider the particular case $F_n\to 0$ in $H^1$. 
As $q_n$ satisfies,  for $ f \in C_0^2(\mathcal I),$
\begin{equation}\label{eq:pvbl}
-\int_{-\pi}^\pi  \frac{ \left(q'_n   f' + \alpha^2 q_n  f\right)(y)}{(Z_--c_n)(Z_+-c_n)}  \,\mathrm{d}y  = \int_{-\pi}^\pi \frac{(F_n  f')(y)}{(Z_--c_n)(Z_+-c_n)} \,\mathrm{d}y. 
\end{equation}
By the fact that 
\beno
\int_{-\pi}^\pi  \frac{\G(y)}{(Z_--c_n)(Z_+-c_n)}  \,\mathrm{d}y=\int_{-\pi}^\pi  \f{1}{2b(y)}\frac{\G(y)}{Z_--c_n}  \,\mathrm{d}y-\int_{-\pi}^\pi  \f{1}{2b(y)}\frac{\G(y)}{Z_+-c_n}  \,\mathrm{d}y,
\eeno
and 
\begin{equation*}
-\int_{-\pi}^\pi \text{Log}\left(Z_+(y) -c_n \right) \partial_y \left( \frac{\G(y)}{2b(y)Z_+^{'}(y)} \right) \mathrm{d} y=\int_{-\pi}^\pi  \f{1}{2b(y)}\frac{\G(y)}{Z_+-c_n}  \,\mathrm{d}y, 
\end{equation*}
we obtain that the right hand side of Equation \eqref{eq:pvbl} can estimated as
\begin{align*}
\left|\int_{-\pi}^\pi \frac{(F_n  f')(y)}{(Z_--c_n)(Z_+-c_n)} \,\mathrm{d}y\right| \lesssim & \left|\int_{-\pi}^\pi  \f{1}{2b(y)}\frac{(F_n  f')(y)}{Z_--c_n}  \,\mathrm{d}y\right| \\
& +\left|\int_{-\pi}^\pi \text{Log}\left(Z_+(y) -c_n \right) \partial_y \left( \frac{(F_n  f')(y)}{2b(y)Z_+^{'}(y)} \right) \mathrm{d} y\right|\lesssim \|F_n\|_{H^1}.
\end{align*}
Passing to the limit as $n \to \infty,$ we can see that the right hand side of Equation \eqref{eq:pvbl} vanishes as $F_n \to 0$ in $H^1,$ and we obtain
\begin{align}\label{eq:res}
0=-\text{p.v.}\int_{-\pi}^\pi  \frac{ \left(q'  f' + \alpha^2 q  f\right)(y)}{(Z_--c)(Z_+-c)}  \,\mathrm{d}y + i\pi \frac{\left( q'  f' + \alpha^2 q  f\right)(y_{c}) }{2b(y_c)Z_+'(y_c)},  \ \forall  f \in C_0^1(\mathcal I).
\end{align}

Let $\chi$ be a smooth cut-off function such that $0 \leq \chi \leq 1,$ $\chi \equiv 1$ near $y_\lambda$ and $\text{supp}\, \chi \Subset \mathcal I.$ Let us specify a test function $ f(y) = \chi(y)\overline{q(y)} b(y_c)Z'_+(y_c).$ The imaginary part of Equation \eqref{eq:res} then reads $ i\pi  \left(|q'(y_{c})|^2+\alpha^2 |q(y_{c})|^2 \right) =0,$ which in turn reveals that $$q(y_{c})=q'(y_{c})=0.$$

By the facts that $\pa_y{\Phi}(y,c)=\f{q(y,c)}{(Z_--c)(Z_+-c)} \sim \frac{q}{y-y_c},$ $q(y_c) =0$ and Hardy's inequality, we have
\beno
\|\pa_y\Phi\|_{L^2(\mathcal I)}\leq C \|q'\|_{L^2(\mathcal I)}.
\eeno 
Thus, we conclude that $\Phi(\cdot,c) \in H^1(\mathcal I)$. 

Multiplying both sides of Equation \eqref{eq:odeihom} by $ f \in H^1_{w, 0}(\mathcal I),$ passing to the limit as $n \to \infty$ and integrating by parts, we obtain the identity in \eqref{eq:ibpmn}.
\end{proof}

We proceed to prove the analogue of Lemma \ref{lem: m} for $y_0 \in (Z_{s})^{-1}(c)$ which is also a critical point of $Z_s,$ $s=+$ or $-.$ There are several scenarios depending on how $c_n$ approaches $c.$ 

(1)  In the case that $c=Z_s(y_0)$ and $Z_s ''(y_0) \left( \text{Re}\,c_n -c \right) \geq 0,$ by the assumptions on $u$ and $b,$ the intersection of $(Z_s)^{-1}(\text{Re}\, c_n)$ with a sufficiently small neighborhood of $y_0$ consists of two points $y_{c_n}^\ell$ and $y_{c_n}^r$ such that $y_{c_n}^\ell \leq  y_0  \leq y_{c_n}^r.$ (If $c =\text{Re}\,c_n,$ equality holds and the two points shrink into one.) Furthermore, it's true that modulo subsequences either $|\text{Re}\,c_n -c| \geq |\text{Im}\,c_n|$ or $|\text{Re}\,c_n -c| \leq |\text{Im}\,c_n|, \forall n \in \mathbb N.$ 

(2)  In the case $Z''_s (y_0) \left( \text{Re}\,c_n -c \right) < 0,$ any small neighborhood of $y_0$ no longer intersects $(Z_s)^{-1}(\text{Re}\, c_n).$ Instead, $\text{Re}\, (Z_s-c_n)$ is sign-definite as $\text{Re}\, (Z_s -c_n) = (Z_s-c) + (c-\text{Re}\,c_n).$

We first consider the case (2) together with the case (1) when $|\text{Re}\,c_n - c| \leq |\text{Im}\,c_n |, \forall n \in \mathbb N.$ 
\begin{lemma}\label{lem: i}
Let $\{c_n\}_{n=1}^\infty \subset \left(\Omega_{\ep_0}\setminus \left(\mathrm{Ran}\,Z_+\cup \mathrm{Ran}\, Z_- \right) \right)$ be such that $c_n \to c$ as $n \to \infty$ for a certain $c \in \left(\mathrm{Ran}\, Z_+ \cup \mathrm{Ran}\,Z_- \right).$ Let the triple $\{c_n,  \Phi_n,  F_n  \}_{n=1}^\infty$ solve \eqref{eq:eln00} and satisfy 
\begin{itemize}
\item $\Phi_n \rightharpoonup \Phi$ in $L^2$ and $(Z_+- c_n)(Z_- -c_n)\partial_y\Phi_n \rightharpoonup (Z_+- c)(Z_- -c) \pa_y\Phi$ in $L^2$ on $\mathcal I_0,$ 
\item $F_n \to F$ in $H^1$ on $\mathcal I_0,$ 
\end{itemize}
for some interval $\mathcal{I}_0 := [y_1,  y_2] \subset \mathbb T$ such that for $s=+$ or $-,$ one of the following situations occurs:
\begin{enumerate}
\item There exists $y_0 \in (Z_s)^{-1}(c)$ in $\mathcal I_0$ at which $Z_s'(y_0) =0;$ $Z''_s(y_0) Z''_s(y) >0, \forall y \in \mathcal I_0.$ For $n$ sufficiently large, $|\mathrm{Re}\, c_n-c|\leq |\mathrm{Im}\, c_n|$ and $Z_s''(y_0) \left( \mathrm{Re}\,c_n -c \right) \geq 0.$
\item There exists $y_0 \in (Z_s)^{-1}(c)$ in $\mathcal I_0$ at which $Z_s'(y_0) =0;$ $Z''_s(y_0) Z''_s(y) >0, \forall y \in \mathcal I_0.$ For $n$ sufficiently large, $Z_s''(y_0) \left( \mathrm{Re}\,c_n -c \right) < 0.$ 
\end{enumerate}
Then for $n$ sufficiently large, 
\begin{enumerate}
\item $(y-y_0)\pa_y\Phi_n\in L^2$, $(y-y_0)\pa_y\Phi\in L^2$ and 
\begin{align}\label{eq:ibp54}
- \int_{\mathcal{I}_0} (Z_+ -c)(Z_--c) \left( \pa_y\Phi f' + \al^2\Phi f \right) \, \mathrm{d}y=\int_{\mathcal I_0} F  f\,\mathrm{d}y, \ \forall  f \in H^1_{w,0},
\end{align}
where $H^1_{w,0} =\left\{ f \in L^2 : \left((y-y_0)\pa_y f \right) \in L^2 \text{ and } f \vert_{y=y_1} =f\vert_{y=y_2} =0 \right\};$
\item there exists a constant $C$ independent of $c_n$ such that 
\begin{equation}\label{esc}
 \left|\Phi_n (y_0) \right | \leq  C \left|(Z_+(y_0) -c_n)(Z_-(y_0) -c_n) \right|^{-\frac{1}{4}};
\end{equation}
\item there exists a constant $C$ independent of $c_n$ such that 
\begin{equation}\label{eslam}
 \left| \pa_y\Phi_n (y_0) \right | \leq  C \left|(Z_+(y_0) -c_n)(Z_-(y_0) -c_n) \right|^{-\frac{3}{4}}.
\end{equation}
\end{enumerate}
\end{lemma}

\begin{proof} We recall that we shall only prove the part of the lemma for $Z_+(y_0)=c,$ as the proof when $Z_-(y_0) =c$ is along the same lines. As $\text{Ran}\,Z_+ \cap \text{Ran}\, Z_- =\emptyset,$ we can find some constant $C>1$ such that $C^{-1} < \left( \text{Re}\, c_n - Z_- \right)  <C,$ provided that $n$ is sufficiently large. Without loss of generality,  we may assume that $Z_+''(y_0) > 0,$ i.e., $Z_+(y_0)$ is a local minimum. 

We can choose $h>0$ small enough independent of $n$ so that $\mathcal{I}_{2h}:= [y_0-2h,  y_0+2h] \subset \mathcal I_0.$ We also denote $\mathcal{I}_{h}:= [y_0-h,  y_0+h].$ Let us introduce a cut-off function $\chi$ satisfying 
\begin{enumerate}
\item $\chi \equiv 1$ on $\mathcal I_{h}$ and $\chi \equiv 0$ outside $\mathcal I_{2h},$
\item $\chi \in C^1$ and $0 \leq \chi \leq 1,$ 
\end{enumerate}
Integrating by parts,  we have
\begin{equation}\notag
\int_{\mathcal I_0} \Phi_n \overline{\Phi_n} \chi \,\mathrm{d}y = -\int_{\mathcal I_0} (y-y_0)\Phi_n \partial_y \left(\overline{\Phi_n} \chi \right)\,\mathrm{d}y -\int_{\mathcal I_0} (y-y_0)\partial_y \Phi_n \overline{\Phi_n}  \chi\, \mathrm{d}y,
\end{equation}
which leads to
\begin{equation}\label{eq:echi2}
\|\sqrt{\chi}\Phi_n \|_{L^2(\mathcal I_{2h})} \leq 2 \|(y-y_0) \partial_y \Phi_n \sqrt{\chi} \|_{L^2(\mathcal I_{2h})} + \|\Phi_n\|_{L^2(\mathcal I_{2h}\setminus \mathcal I_{h})}.
\end{equation}

Multiplying Equation \eqref{eq:eln00} by $\overline{\Phi_n} \chi,$ integrating by parts and taking the real part and the imaginary part separately, we obtain the following identities 
\begin{equation}\label{eq:rchi} 
\begin{split}
& -\int_{\mathcal I_0} \left( (Z_+ -\text{Re}\,c_n)(Z_- - \text{Re}\, c_n) -(\text{Im}\,c_n)^2 \right)\left( |\partial_y \Phi_n|^2 +\alpha^2 |\Phi_n|^2 \right) \chi \,\mathrm{d}y\\
&=  \text{Re}\,\langle F_n, \chi \overline{\Phi_n}  \rangle_{L^2(\mathcal I_0)} + \text{Re}\,\int_{\mathcal I_{2h}\setminus \mathcal I_h} (Z_+ -c_n)(Z_- -c_n)\overline{\Phi_n}  \partial_y \Phi_n \,\chi' \mathrm{d}y,
\end{split}
\end{equation}
\begin{equation}\label{eq:ichi}
\begin{split}
& - \int_{\mathcal I_0} \text{Im}\,c_n (Z_+ + Z_- -2 \text{Re}\, c_n) \left( |\partial_y \Phi_n|^2 +\alpha^2 |\Phi_n|^2 \right) \chi\,\mathrm{d}y\\
 &=   \text{Im}\,\langle F_n, \chi\overline{\Phi_n}  \rangle_{L^2(\mathcal I_0)} + \text{Im}\,\int_{\mathcal I_{2h}\setminus \mathcal I_h} (Z_+ -c_n)(Z_- -c_n)\overline{\Phi_n}  \partial_y \Phi_n \chi'\,\mathrm{d}y.
\end{split} 
\end{equation}

To estimate \eqref{eq:ichi}, we note that for $h$ small enough and for $n$ large enough, it holds that 
$$|Z_+(y)-\mathrm{Re}\, c_n|\leq Ch+\mathrm{dist}(c_n,\mathrm{Ran}\,Z_+)\leq \ep_0+Ch \leq Ch, \ \forall y \in \mathcal I_{2h},$$ 
which leads to the following bound:
\begin{align*}
Z_+ + Z_- -2 \text{Re}\, c_n \leq & Z_-(y)-\text{Re}\, c_n+|Z_+-\text{Re}\, c_n| \leq -C^{-1}+Ch \leq -\f12C^{-1}.
\end{align*}
Hence, it follows from \eqref{eq:ichi} that
\begin{equation}\label{eichi}
\begin{split}
& |\text{Im}\,c_n| \left(\|\sqrt{\chi}\pa_y\Phi_n\|_{L^2(\mathcal I_{2h})}^2+\alpha^2 \|\sqrt{\chi}\Phi_n\|_{L^2(\mathcal I_{2h})}^2\right)\\
\lesssim &   \int_{\mathcal I_{2h}} \left|F_n \chi\overline{ \Phi_n} \right| \mathrm{d}y + \|(Z_+ -c_n)(Z_- -c_n)\pa_y\Phi_n\|_{L^2(\mathcal I_{2h}\setminus \mathcal I_h)}^2.
\end{split}
\end{equation}
In particular, \eqref{eichi} shows that $\left\| \sqrt{|\text{Im}\, c_n|} \partial_y \Phi_n \right\|_{L^2(\mathcal I_{h})} <C.$

As $\left( Z_+(y) -\text{Re}\,c_n \right)  =  \left( Z_+(y) -Z_+(y_0) + c -\text{Re}\, c_n \right),$ we can write \eqref{eq:rchi} as
\begin{equation}\label{sad}
\begin{split}
& \int_{\mathcal I_{2h}} \left(Z_+(y) -Z_+(y_0) \right)( \text{Re}\,c_n- Z_- ) \left( |\partial_y \Phi_n|^2 + \alpha^2|\Phi_n|^2 \right) \chi \,\mathrm{d}y \\
= & \text{Re}\,\langle F_n, \chi \overline{\Phi_n}  \rangle_{L^2(\mathcal I_0)} + \text{Re}\,\int_{\mathcal I_{2h}\setminus \mathcal I_h} (Z_+ -c_n)(Z_- -c_n)\overline{\Phi_n}  \partial_y \Phi_n \,\chi' \mathrm{d}y \\
& + \int_{\mathcal I_{2h}} \left( \left(\text{Re}\,c_n- c \right)( \text{Re}\,c_n- Z_- ) -|\text{Im}\, c_n|^2 \right) \left( |\partial_y \Phi_n|^2 + \alpha^2|\Phi_n|^2 \right)\chi \,\mathrm{d}y.
\end{split}
\end{equation}

For the case $|\text{Re}\,c_n-c|\leq | \text{Im}\, c_n|$ and $Z_+''(y_0) \left( \text{Re}\,c_n -c \right) \geq 0,$ we notice that the term on right hand side of \eqref{sad} containing $(\text{Re}\,c_n-c)$ can be dominated by the estimate of \eqref{eq:ichi}. Hence, we have
\begin{equation}\notag
\begin{split}
& \int_{\mathcal I_{2h}} \left(Z_+(y) -Z_+(y_0) \right)( \text{Re}\,c_n- Z_- ) |\partial_y \Phi_n|^2\chi \,\mathrm{d}y \\
\lesssim  & \left| \text{Re}\,\int_{\mathcal I_{2h}} F_n \overline{ \Phi_n}  \mathrm{d}y \right| + \left| \text{Re}\,\int_{\mathcal I_{2h}\setminus \mathcal I_h} (Z_+ -c_n) (Z_- -c_n)\overline{ \Phi_n}\partial_y \Phi_n \,\chi'  \mathrm{d}y\right|\\
& +\left| \text{Im}\,c_n \right| \int_{\mathcal I_{2h}} \left( \text{Re}\,c_n-Z_-  \right) \left( |\partial_y \Phi_n|^2+ \alpha^2 |\Phi_n|^2 \right) \mathrm{d}y\\
\lesssim &  C \int_{\mathcal I_{2h}} \left|F_n \chi\overline{ \Phi_n} \right| \mathrm{d}y  + C \int_{\mathcal I_{2h}\setminus \mathcal I_h} \left|  (Z_+ -c_n ) ( Z_- -c_n  )\overline{\Phi_n} \partial_y \Phi_n \right| \mathrm{d}y,\\
\end{split}
\end{equation}
whereas in the case $Z_+''(y_0) \left( \text{Re}\,c_n -c \right) < 0,$ $\left(c -\text{Re}\, c_n \right)$ has the same sign as $\left( Z_+(y) - c \right),$ from which we infer
\begin{equation}\notag
\begin{split}
& \int_{\mathcal I_{2h}}  \left(Z_+(y) -Z_+(y_0) \right)(\text{Re}\, c_n - Z_-) |\partial_y \Phi_n|^2\chi \,\mathrm{d}y \\
 \leq &\int_{\mathcal I_{2h}}  \left| \chi F_n \overline{ \Phi_n} \right| \mathrm{d}y  + \left| \int_{\mathcal I_{2h}\setminus \mathcal I_h} (Z_+ -c_n)(Z_- -c_n) \overline{ \Phi_n} \partial_y \Phi_n \,\chi'  \mathrm{d}y\right|.
\end{split}
\end{equation}

Using the fact that $(Z_+ -c) \sim (y-y_0)^2$ in $\mathcal I_0,$ we have,  by \eqref{eq:echi2} and \eqref{eq:rchi},
\begin{equation}\label{eq:bdd1}
\begin{split}
\| (y-y_0) \partial_y \Phi_n \sqrt{\chi}\|_{L^2(\mathcal I_{2h})}^2 \leq &  C\left( \|\Phi_n \|_{L^2(\mathcal I_{2h} \setminus \mathcal I_h)}^2 + \|(Z_+ -c_n) \pa_y \Phi_n \|_{L^2(\mathcal I_{2h} \setminus \mathcal I_h)}^2 \right)\\
&+ C\|F_n \|_{L^2(\mathcal I_0)}^2.
\end{split}
\end{equation}

Multiplying both sides of Equation \eqref{eq:eln00} by $ f \in H^1_{0}(\mathcal I_0)$ and passing to the limit as $n \to \infty,$ we have
\begin{align*}
-\int_{\mathcal{I}_0} (Z_+ -c)(Z_- -c) \left( \pa_y\Phi f' + \al^2\Phi f\right) \,\mathrm{d}y= \int_{\mathcal I_0} F  f \,\mathrm{d}y, \ \forall  f \in H^1_0(\mathcal I_0).
\end{align*}
Estimate \eqref{eq:bdd1} ensures that  $(y-y_0)\pa_y\Phi\in L^2$
and
\beq\label{eq:Phi_n-L^2}
\|\Phi_n\|_{L^2(\mathcal I_0)}\lesssim \|(Z_+ -c_n)(Z_--c_n) \pa_y \Phi_n \|_{L^2(\mathcal I_{2h} \setminus \mathcal I_h)}+\|\Phi_n\|_{L^2(\mathcal I_{2h} \setminus \mathcal I_h)}+\|F_n\|_{L^2}.
\eeq 
Thus, by the density of $H^1_0(\mathcal I_0)$ in $H^1_{0,w}(\mathcal I_0),$ Equation \eqref{eq:ibp54} holds.

It follows from Gargliardo-Nirenberg inequality and the fact that $\left \|\sqrt{\text{Im}\,c_n} \pa_y \Phi_n \right\|_{L^2} < C$ that 
\begin{align*}
|\Phi_n(y_0)| \leq \|\Phi_n\|_{L^\infty} \lesssim \|\pa_y \Phi_n\|_{L^2}^{\frac{1}{2}} \|\Phi_n\|_{L^2}^{\frac{1}{2}} \lesssim |\text{Im}\,c_n|^{-\frac{1}{4}} \lesssim |Z_+(y_0)-c_n|^{-\frac{1}{4}}.
\end{align*}
Thus, Estimate \eqref{esc} holds.

As for Estimate \eqref{eslam}, we can infer from the fact that $(y -y_0) \partial_y \Phi_n \in L^2$ the existence of $\wt y \in [y_0-2\d_n,y_0-\d_n]$ such that $\left| \partial_y \Phi_n (\wt y) \right| < C\d_n^{-\frac{3}{2}}.$ We further set $\d_n=\sqrt{\left| c-c_n \right|}.$

Since $q_n \in H^2,$ as shown in \eqref{es2}, and $F_n \in H^1,$ we can integrate \eqref{eq:eln00} from $\wt y$ to $y_0,$ which results in 
\begin{equation}\label{strctr}
\begin{split}
&(Z_+(y_0)-c_n)(Z_-(y_0)-c_n) \partial_y \Phi_n (y_0)\\ 
= & \left( (Z_+ (\wt y)-c)+ (c- c_n) \right) (Z_-(\wt y)-c_n) \partial_y \Phi_n (\wt y)\\
& + \int_{\wt y}^{y_0} F_n(y) \,\mathrm{d}y + \alpha^2 \int_{\wt y}^{y_0} (Z_+(y)-c_n)(Z_-( y)-c_n) \Phi_n (y)\, \mathrm{d}y.
\end{split}
\end{equation}
By Equation \eqref{strctr} and the fact that $|Z_+(\wt y) -c| \leq C \d_n^2 =C|c- c_n|,$ we have
\begin{equation}
|\partial_y \Phi_n(y_0) | \leq  C \frac{\left| c- c_n \right|}{\left| Z_+(y_0)-c_n  \right| \d_n^{\frac{3}{2}}} + \frac{\d_n}{|Z_+(y_0)-c_n |} \leq  \frac{C}{ \left| Z_+(y_0) -c_n \right|^{\frac{3}{4}} }.
\end{equation}
\end{proof}

We then address the scenario when $y_0 \in (Z_s)^{-1}(c)$ with $Z''_s (y_0) \left( \text{Re}\,c_n -c \right) \geq 0,$ $s=+$ or $-,$ and $|\text{Re}\, c_n-c|\geq |\text{Im}\, c_n|.$ We recall that $y_{c_n}^\ell$ and $y_{c_n}^r$ are the two points in a sufficiently small neighborhood of $y_0$ such that $y_{c_n}^\ell \leq y_0 \leq y_{c_n}^r$ and $\text{Re}\,c_n = Z_s \left(y_{c_n}^\ell \right) = Z_s (y_{c_n}^r).$

\begin{lemma}\label{lem: r}
Let $\{c_n\}_{n=1}^\infty \subset \left(\Omega_{\ep_0}\setminus \left(\mathrm{Ran}\,Z_+\cup \mathrm{Ran}\, Z_- \right) \right)$ be such that $c_n \to c$ as $n \to \infty$ for a certain $c \in \left(\mathrm{Ran}\, Z_+ \cup \mathrm{Ran}\,Z_- \right).$ Let the triple $\{c_n,  \Phi_n,  F_n  \}_{n=1}^\infty$ solve \eqref{eq:eln00} and satisfy 
\begin{itemize}
\item $\Phi_n \rightharpoonup \Phi$ in $L^2$ and $(Z_+- c_n)(Z_- -c_n)\partial_y\Phi_n \rightharpoonup (Z_+- c)(Z_- -c) \pa_y\Phi$ in $H^1$ on $\mathcal I_0,$ 
\item $F_n \to F$ in $H^1$ on $\mathcal I_0,$ 
\end{itemize}
for some interval $\mathcal{I}_0:= [y_1,  y_2] \subset \mathbb T$ such that there exists $y_0 \in (Z_s)^{-1}(c)$ in $\mathcal I_0$ at which $Z'_s(y_0)=0;$ $Z_s''(y_0) Z''_s(y) >0, \forall y \in \mathcal I_0;$ $|\mathrm{Re}\, c_n-c|\geq  |\mathrm{Im}\, c_n|$ for sufficiently large $n$ and $Z_s''(y_0) \left( \mathrm{Re}\,c_n -c \right) \geq 0,$ where $s=+$ or $-.$

Then $(y-y_0)\pa_y\Phi\in L^2$ and 
\begin{align}\label{eq:55}
-\int_{\mathcal{I}_0} (Z_+ -c)(Z_- -c) \left(\pa_y\Phi f' +\al^2 \Phi f \right) \, \mathrm{d}y=\int_{\mathcal I_0} F  f\,\mathrm{d}y, \ \forall  f \in H^1_{w,0},
\end{align}
where $H^1_{w,0} =\left\{ f \in L^2 : \left((y-y_0)\pa_y f\right) \in L^2 \text{ and } f\vert_{y=y_1} =f\vert_{y=y_2} =0 \right\}.$
\end{lemma}

\begin{proof}  As in Appendix \ref{sturm},  we can construct a solution, which we denote by $\varphi^\ell_n$ (or $\varphi^r_n$), to the homogeneous equation 
\begin{align}\label{eq:hmg}
& \partial_y ((Z_+ -c_n)(Z_- -c_n )\partial_y \varphi) -\alpha^2(Z_+ -c_n)(Z_- -c_n)\varphi =0,  \\
& \varphi \left(y_{c_n}^\ell \right) =1,  \ \partial_y \varphi \left(y_{c_n}^\ell \right) =0 \  \left (\text{or } \varphi \left(y_{c_n}^r \right) =1,  \ \partial_y \varphi \left(y_{c_n}^r \right) =0 \right),
\end{align}
on an interval $[y_1,  y_0] $ (or $(y_0, y_2],$ respectively).

We note that $y_{c_n}^i \to y_0$ as $c_n \to c,$ $i=\ell$ or $r.$ In turn, we denote by $\varphi^\ell$ (or $\varphi^r$) the solution to 
\begin{align*}
&\partial_y ((Z_+ -c)(Z_- -c)\partial_y \varphi) -\alpha^2(Z_+ -c)(Z_- -c)\varphi =0, \\
& \varphi \left(y_{0} \right) =1,  \ \partial_y \varphi \left(y_0 \right) =0,
\end{align*}
on $[y_1,  y_0]$ (or $[y_0, y_2],$ respectively). 

We shall use the following properties of $\varphi_n^i$ and $\varphi^i$, $i= \ell$ or $r,$ as proven in Appendix \ref{sturm}:
\begin{enumerate}
\item $\|\pa_y\varphi^{i}_n\|_{L^{\infty}}+\|\varphi^{i}_n\|_{L^{\infty}}  \leq C$ and $\|\pa_y\varphi^{i}\|_{L^{\infty}}+\|\varphi^{i}\|_{L^{\infty}} \leq C,$
\item $\left|\varphi_n^{i} \right| \geq \frac{1}{2}$ and $\varphi^i\geq 1$,
\item $\left|\varphi_n^{i}(y) -1 \right|\leq C\left |y-y^i_{c_n} \right|^2, \ y \in [y_1, y^\ell_{c_n}]$ (or $[y^r_{c_n}, y_2],$ respectively),\\
and $\left|\varphi^{i}(y) -1 \right| \leq C\left |y-y_0 \right|^2,$ $y \in [y_1, y_0]$ (or $[y_0, y_2],$ respectively),  
\item $\lim\limits_{n \to \infty}\varphi_n^{i}=\varphi^{i}.$
\end{enumerate}

For $y \in [y_1,   y_{c_n}^\ell]$ we can solve the inhomogeneous equation \eqref{eq:eln00} explicitly by
\begin{equation}\label{eq:phin-}
\begin{split}
\Phi_n(y) = & \frac{\varphi_n^\ell (y)}{\varphi_n^\ell (y_1)}\Phi_n(y_1) + \mu_n^{\ell} \varphi_n^\ell (y) \int_{y_1}^y \frac{1}{\left((Z_+ -c_n)(Z_- -c_n)(\varphi_n^\ell )^2\right)(y')} \mathrm{d}y'\\
& + \varphi_n^\ell (y) \int_{y_1}^{y} \frac{\int_{y_{c_n}^\ell}^{y'} \left(F_n \varphi_n^\ell \right)(y'') \mathrm{d}y'' }{ \left( (Z_+ -c_n)(Z_- -c_n) (\varphi^\ell_n)^2\right) (y')} \mathrm{d}y',
\end{split}
\end{equation}
as \eqref{eq:eln00} is equivalent to
\begin{equation}
\partial_y \left((Z_+ -c_n)(Z_- -c_n) \left(\varphi^{i}_n \right)^2 \partial_y \left(\frac{\Phi_n}{\varphi_n^{i}} \right) \right) = F_n \varphi_n^{i}, \  i = \ell \text{ or } r.
\end{equation}
Here the coefficient $\mu_n^{\ell}$ is given by
$$\mu_n^{\ell} := (Z_+(y_1)-c_n)(Z_-(y_1)-c_n)\left( (\varphi^\ell_n \Phi'_n)(y_1) - \left(\left( \varphi^\ell_n \right)'\Phi_n \right)(y_1) \right)- \int_{y_{c_n}^\ell}^{y_1} \left( F_n \varphi_n^\ell \right) (y) \mathrm{d}y.$$

We have for the last term in \eqref{eq:phin-} that 
\begin{equation}\label{eq:l22}
\begin{split}
\left|  \varphi_n^\ell (y) \int_{y_1}^y \frac{\int_{y^\ell_{c_n}}^{y'} \left(F_n \varphi_n^\ell \right)(y'') \mathrm{d}y''}{\left((Z_+ -c_n)(Z_- -c_n)(\varphi_n^\ell)^2\right)(y')} \mathrm{d}y' \right| \leq & C \left| \int_{y_1}^y \frac{1 }{\left(y' + y^\ell_{c_n} \right)} \mathrm{d}y'  \right| \\
 \leq & C \ln\left(\frac{y+y^\ell_{c_n}}{y_1+y^\ell_{c_n}} \right) \in L^2(y_1, y^\ell_{c_n}).
\end{split}
\end{equation}
For the second term on the right hand side of \eqref{eq:phin-},  we have the point-wise limit for $y\in [y_1, y_0),$ 
$$\varphi_n^\ell (y) \int_{y_1}^y \frac{1}{\left((Z_+ -c_n)(Z_- -c_n)(\varphi_n^\ell )^2\right)(y')} \mathrm{d}y' \rightarrow \varphi^\ell (y)\int_{y_1}^y \frac{1}{\left((Z_+ -c)(Z_- -c)(\varphi^\ell)^2 \right)(y')} \mathrm{d}y'.$$ The fact that $\text{Re}\,(Z_+ -c) \sim (y- y_0)^2$ implies that 
\begin{align}
\label{eq:l21} \varphi^\ell (y)\int_{y_1}^y \frac{1}{\left((Z_+ -c)(Z_- -c)(\varphi^\ell)^2 \right)(y')} \mathrm{d}y' \sim \frac{1}{(y-y_0)} \not \in L^2(y_1, y_0).
\end{align}

On the other hand, from the facts that $\| \Phi_n\|_{L^2} \leq C,$ $\| (Z_+ -c_n)(Z_- -c_n)\pa_y \Phi_n\|_{H^1} \leq C$ and $\| F_n \|_{L^\infty} < C,$ we know that $\| (Z_+ -c_n)(Z_- -c_n) \Phi_n\|_{H^1} \leq C$ and $\left\{\mu^\ell_n \right\}_{n=1}^\infty$ is bounded. Thus, we further infer that up to a subsequence $\lim\limits_{n \to \infty} \mu_n^\ell = \mu ^\ell$ and that the second term on the right hand side of \eqref{eq:phin-} converges pointwise. Similarly,  we can confirm the pointwise convergence of the third term on the right hand side of \eqref{eq:phin-}. We can verify that the weak limit $\Phi$ can be written as
\begin{equation}\label{eq:phii-}
\begin{split}
\Phi(y) := & \frac{\varphi^\ell (y)}{\varphi^\ell(y_1)}\Phi(y_1) + \mu^{\ell} \varphi^\ell (y) \int_{y_1}^y \frac{1}{\left((Z_+ -c)(Z_- -c)(\varphi^\ell)^2\right)(y')} \mathrm{d}y'\\
& + \varphi^\ell (y) \int_{y_1}^{y} \frac{\int_{y_{0}}^{y'} \left(F \varphi^\ell \right)(y'') \mathrm{d}y'' }{ \left( (Z_+ -c)(Z_- -c) (\varphi^\ell)^2\right) (y')} \mathrm{d}y',\quad y<y_0.
\end{split}
\end{equation}

Moreover,  the assumption $\Phi_n \rightharpoonup \Phi$ in $L^2$ along with \eqref{eq:l22} indicates that the second term on the right hand side of \eqref{eq:phii-} is uniformly bounded in $L^2.$ According to \eqref{eq:l21},  it is therefore necessary that 
$$\mu^\ell =0, \text{ i.e.,  } \mu_{n}^\ell \rightarrow 0 \text{ up to a subsequence}.$$ 

Similarly,  for $y \in [y_{c_n}^r,  y_2]$ we have
\begin{equation}\notag
\begin{split}
\Phi_n(y) = & \frac{\varphi_n^r(y)}{\varphi_n^r(y_2)}\Phi_n(y_2) + \mu_n^{r} \varphi_n^r(y) \int_{y_2}^y \frac{1}{\left((Z_+ -c_n)(Z_- -c_n)(\varphi_n^r)^2\right)(y')} \mathrm{d}y'\\
& + \varphi_n^r (y) \int_{y_2}^{y} \frac{\int_{y_{c_n}^r}^{y'} \left(F_n \varphi_n^r \right)(y'') \mathrm{d}y'' }{ \left( (Z_+ -c_n)(Z_- -c_n) (\varphi^r_n)^2\right) (y')} \mathrm{d}y',
\end{split}
\end{equation}
with 
$$\mu_n^{r}:= (Z_+(y_2)-c_n)(Z_-(y_2)-c_n)\left( (\varphi^r_n \Phi'_n)(y_2) - \left(\left( \varphi^r_n \right)'\Phi_n \right)(y_2) \right)- \int_{y_{c_n}^r}^{y_2} \left( F_n \varphi_n^r \right) (y) \mathrm{d}y,$$
and we can show that $\mu_n^r \rightarrow 0$ up to a subsequence. 

Thus, from the formula \eqref{eq:phii-} and its analogue on $[y^r_{c_n}, y_2],$ where $\mu^i \equiv 0$ for $i = \ell$ or $r,$ we know that 
\begin{align*}
|(y-y_0) \partial_y \Phi| \lesssim & |\Phi(y_1)|\|\pa_y\varphi^i(y)\|_{L^{\infty}}+\|F\|_{L^{\infty}}\|\pa_y\varphi^i(y)\|_{L^{\infty}}|y-y_0|(\ln |y-y_0|+1)\\
& +\|F\|_{L^{\infty}} \|\varphi^i(y)\|_{L^{\infty}}.
\end{align*}

As $\Phi_n \rightharpoonup \Phi$ in $L^2$ and $(Z_+ -c_n)(Z_- -c_n)\pa_y \Phi_n \rightharpoonup (Z_+ -c)(Z_- -c)\pa_y \Phi$ in $H^1,$ multiplying Equation \eqref{eq:eln00} by $ f \in H^1_0(\mathcal I_0)$ and passing to the limit as $n \to \infty,$ we have
\begin{align*}
-\int_{\mathcal{I}_0} (Z_+ -c)(Z_- -c)\left( \pa_y\Phi f' + \al^2\Phi f \right) \,\mathrm{d}y=\int_{\mathcal I_0} F f\,\mathrm{d} y,  \ \forall  f \in H^1_0(\mathcal I_0).
\end{align*}
The desired result \eqref{eq:55} then follows from a density argument.
\end{proof}

In view of Lemma \ref{lem: m},  Lemma \ref{lem: i} and Lemma \ref{lem: r},  the use of $\text{sgn}(Z_+ -c)(Z_- -c) \Phi$ as a test function and integration by parts can be justified. As $F_n \to 0$ in $H^1$ as $n \to \infty,$ it is not difficult to see that 
$$\Phi_n \rightharpoonup  0 \text{ in } L^2(\mathbb{T}) \text{ and }
(Z_+ -c_n)(Z_- -c_n) \pa_y \Phi_n \rightharpoonup 0 \text{ in } H^1(\mathbb{T}) .$$
We shall rigorously prove this claim in Section \ref{proof51}.  To complete the proof of Proposition \ref{prop: lap},  it remains to be shown that $\{ \Phi_n \}_{n=1}^\infty$ and $\{ (Z_+ -c_n) (Z_- -c_n)\pa_y\Phi_n \}_{n=1}^\infty$ converges strongly in $L^2$ and $H^1,$ respectively, throughout $\mathbb T$, which amounts to proving the strong convergence of $\{\Phi_n \}_{n=1}^\infty$ and $\{ (Z_+ -c_n) (Z_- -c_n)\pa_y\Phi_n \}_{n=1}^\infty$ near the critical points in $(Z_+)^{-1}(c) \cup(Z_-)^{-1}(c) $. 
Indeed, away from $(Z_+)^{-1}(c) \cup(Z_-)^{-1}(c),$ it is clear that
\begin{gather}\label{eq: strongto0}
\Phi_n \to  0 \text{ in }  H^1_{loc}\left(\mathbb{T}\setminus ((Z_+)^{-1}(c) \cup(Z_-)^{-1}(c)) \right),\\
(Z_+ -c_n)(Z_- -c_n) \pa_y \Phi_n \to 0 \text{ in } H^1_{loc} \left(\mathbb{T}\setminus ((Z_+)^{-1}(c) \cup(Z_-)^{-1}(c)) \right),
\end{gather}
and the case of points in $(Z_\pm)^{-1}(c)$ at which $|Z_\pm' | >0$ is already covered in Lemma \ref{lem: m}.

\subsection{Strong convergence of $\{\Phi_n \}_{n=1}^\infty$ near $y_0$}

As for strong convergence of $\{ \Phi_n \}_{n=1}^\infty$ and $\{ (Z_+ -c_n) (Z_- -c_n)\pa_y\Phi_n \}_{n=1}^\infty$ when $c_n\to c=Z_s(y_0)$ with $Z_s '(y_0)=0,$ $s=+$ or $-,$ we note that for the scenarios in Lemma \ref{lem: i}, by \eqref{eq:Phi_n-L^2} and \eqref{eq: strongto0}, it is not difficult to show that as $F_n\to 0$ in $H^1$, $\Phi_n \to 0$ and $(Z_+ -c_n)(Z_- -c_n) \pa_y \Phi_n \rightarrow 0$ in $L^2(\mathcal I_0)$ and $H^1(\mathcal I_0),$  respectively via integration by parts, while a more delicate analysis is needed for the scenario in Lemma \ref{lem: r}, i.e.,  $Z_s''(y_0) \left( \text{Re}\,c_n -c \right) \geq 0,$ $s=+$ or $-,$ and $|\text{Re}\, c_n-c |\geq |\text{Im}\, c_n|.$ We shall postpone the proof for the scenarios in Lemma \ref{lem: i} to Section \ref{proof51} and focus primarily on that for the scenario in Lemma \ref{lem: r}. 

We shall prove the following lemma.
\begin{lemma}\label{stgcv} 
Let $y_0 \in \left( (Z_+)^{-1}(c) \cup (Z_-)^{-1}(c) \right)$ be a critical point, i.e., $Z_s'(y_0) =0$, and the interval $\mathcal I_0: =[y_1, y_2]$ be such that $y_0 \in \mathcal I_0$ and $ Z_s''(y_0)Z_s''(y)>0,$ $s=+$ or $-.$ Let $\{ (c_n, \Phi_n, F_n  ) \}_{n=1}^\infty$ and $\Phi$ satisfy the conditions as in Lemma \ref{lem: r}, then $\Phi_n \rightarrow \Phi$ in $L^2$ in $\mathcal I_0.$ In addition, the following estimates hold
\begin{eqnarray}
& \label{f0} |\Phi_n(y_0)| \leq C|(Z_-(y_0)-c_n)(Z_+(y_0)-c_n)|^{-\frac{1}{4}},\\
& \label{pf0} |\pa_y \Phi_n(y_0) | \leq  C |(Z_-(y_0)-c_n)(Z_+(y_0) -c_n)|^{-\frac{3}{4}}.
\end{eqnarray}
\end{lemma}

To prove Lemma \ref{stgcv}, we shall construct solutions to Equation \eqref{eq:eln00} in $\mathcal I_0$ and obtain its explicit formula using ODE techniques. 

Without loss of generality, we may assume that the critical point $y_0 =0.$ Let $\mathcal I_0:= [y_1, y_2]$ be such that $Z_+ (y_1)=Z_+(y_2)$ (or $Z_- (y_1)=Z_-(y_2)$) with $y_1 < 0 < y_2.$ We shall rewrite Equation \eqref{eq:eln00} by introducing 
\begin{equation}\notag
 \Phi_n^* (y):= \Phi_n(y) -\mathcal L_{\Phi_n}(y):= \Phi_n(y) - \frac{\Phi_n(y_2) -\Phi_n(y_1)}{(y_2 -y_1)}(y-y_1) - \Phi_n(y_1),
\end{equation}
which leads to
\begin{equation}\label{eq:2}
\begin{cases}
\partial_y\left( (Z_+ -c_n)(Z_- -c_n)\partial_y \Phi^*_n \right) -\alpha^2 (Z_+ - c_n)(Z_- -c_n) \Phi^*_n = F^*_n,\\
\Phi^*_n (y_1) = \Phi^*_n(y_2) =0,
\end{cases}
\end{equation}
with
\begin{equation}\notag
\begin{split}
F^*_n(y) = & F_n(y) -  \frac{\Phi_n(y_2) -\Phi_n(y_1)}{(y_2 -y_1)} (Z_+')(y)(Z_-(y) -c_n)\\
& - \frac{\Phi_n(y_2) -\Phi_n(y_1)}{(y_2 -y_1)} (Z_-')(y)(Z_+(y) -c_n) + \alpha^2 (Z_+ -c_n)(Z_- -c_n) \mathcal L_{\Phi_n} (y).
\end{split}
\end{equation}

We can see that to study Equation \eqref{eq:2} is to study an equation of the following type
\begin{equation}\label{eq:blan}
\begin{cases}
\partial_y\left( (Z_+-c_* )(Z_--c_* ) \partial_y\Phi_* \right) -\alpha^2(Z_+-c_* )(Z_--c_* ) \Phi_* = F_*,\\
\Phi_* (y_1) = \Phi_* (y_2) =0.
\end{cases}
\end{equation}
Here $c_* \in \left( \Omega_{\ep_0} \setminus \left( \text{Ran}\, Z_+ \cup \text{Ran}\, Z_- \right)\right)$ satisfies the conditions $\left| \text{Re}\,c_* -Z_+(0) \right| \geq \left| \text{Im} \,c_* \right|$ and $Z_+''(0)(\text{Re}\, c_* - Z_+(0) ) \geq 0$ (or $\left| \text{Re}\,c_* -Z_-(0) \right| \geq \left| \text{Im} \,c_* \right|$ and $Z_-''(0)(\text{Re}\, c_* - Z_-(0) ) \geq 0$), while $F_* \in H^1.$ For $c^*$ close enough to $Z_+ (0)$ (or $Z_-(0)$), we can find exactly two points $y^i_{*},$ $i=\ell$ or $r,$ in a sufficiently small neighborhood of $0$ such that $y_{*}^\ell \leq 0 \leq y^r_{*}$ and $Z_+ \left(y_{*}^\ell \right) =Z_+(y_{*}^r) = \text{Re}\,c^*$ (or $Z_- \left(y_{*}^\ell \right) =Z_+(y_{*}^r) = \text{Re}\,c_*$). It turns out that the solution to Equation \eqref{eq:blan} enjoys the following estimate.
\begin{lemma}\label{lem: phi*} 
Let $\Phi_*$ be the solution to Equation \eqref{eq:blan}. There exists some $C >0$ independent of $c_*$ such that 
\begin{equation}\notag
\|\Phi_*\|_{L^2(y_1,y_2)} \leq C \|F_*\|_{L^\infty}. 
\end{equation}
\end{lemma}
 
Before we prove Lemma \ref{lem: phi*}, we make some preparations. Recall that here we assume that $0$ is a critical point of $Z_+$, i.e. $Z_+'(0)=0$ (or $Z_-'(0)=0$) and that we restrict ourselves to the case $c_* \to Z_+(0)$ with $\text{Im}\, c_* >0$, as the proof for the other cases  are along the same lines. 

Without loss of generality, we may also assume that $Z_+''(0) > 0,$ i.e., $Z_+(0)$ is a local minimum. We define $\sigma(c_*) \in \mathbb C,$ which satisfies $\left( \sigma(c_*) \right)^2=c_*-Z_+(0)$ with $\text{Im}\,\sigma(c_*)>0.$ 

To this end, we proceed to introduce the notations. 
We define the function $V(y)$ such that $\left(V(y) \right)^2 = Z_+(y)-Z_+(0),$ i.e.,
\begin{equation}
V(y) = \begin{cases} - \sqrt{Z_+(y)-Z_+(0)} \text{ on } [y_1, 0], \\
\sqrt{Z_+(y) - Z_+(0)} \text{ on } [0, y_2].
\end{cases}
\end{equation}
It can be verified via direct computations that $V \in C^2(\mathcal I_0)$ is monotone, and $V'(0)=\f{\sqrt{2Z''_+(0)}}{2}$.

Denoting the solution to Equation \eqref{eq:blan} on $[y_1, 0]$ by $\Phi_*^\ell$ and that on $[0, y_2]$ by $\Phi_*^r,$ we notice that Equation \eqref{eq:blan} is equivalent to
\begin{equation}
\partial_y\left((Z_--c_*)(Z_+-c_*)\left(\varphi_*^{j} \right)^2 \partial_y\left( \frac{\Phi^{j}_*}{\varphi_*^{j}} \right)  \right) = F_* \varphi_*^{j}, \ j= \ell \text{ or }r,
\end{equation}
where $\varphi_*^r$ and $\varphi_*^\ell$ are the solutions to the homogeneous equation
\begin{eqnarray}
& \partial_y\left((Z_--c_*)(Z_+-c_*)\partial_y\varphi_* \right) -\alpha^2 (Z_--c_*)(Z_+-c_*)\varphi_* = 0,\\
\label{b1} &  \varphi_* \left(y_{*}^\ell \right) =1,  \ \partial_y\varphi_* \left(y_{*}^\ell \right) =0, \ \text{for }y \in [y_1,  0], \\
\label{b2} & \text{or }\varphi_* \left(y_{*}^r \right) =1,  \ \partial_y\varphi_* \left(y_{*}^r \right) =0,  \ \text{for }y \in [0, y_2],
\end{eqnarray}
with $\varphi_*^r$ and $\varphi_*^\ell$ corresponding to condition \eqref{b1} and corresponding to \eqref{b2}, respectively. The following properties of $\varphi_*^{j},$ $j=\ell$ or $r,$ shall be useful -- 
\begin{enumerate}
\item $|\varphi_*^{j}| > \frac{1}{2},$
\item $|\varphi_*^{j}(y) -1 | \leq C \left|y - y_{*}^j \right|^2, \ y \in [y_1, 0]$ (or $[0, y_2]$),
\item $|\partial_y \varphi_*^{j}(y)| \leq C \left |y-y_{*}^j \right|,  \ y \in [y_1, 0]$ (or $[0, y_2]$).
\end{enumerate}

We can integrate twice and obtain explicit solution formulae to Equation \eqref{eq:blan}. On $[0, y_2],$ the solution to Equation \eqref{eq:blan} is given by
\beq\label{eq: Theta_0+}
\begin{split}
\Phi_*^r(y)= & \nu^r[F_*](c_*) \varphi_*^r(y) + {\mu}^{r}[F_*](c_*)  \varphi_*^r(y) \int_0^y \frac{1}{(Z_-(y')-c_*)(Z_+(y')-c_*)\left( \varphi_*^r(y') \right)^2} \mathrm{d}y'\\
&+ \varphi_*^r(y) \int_0^y \frac{\int_{y_{*}^r}^{y'}(F_* \varphi_*^r)(z)\mathrm{d}z}
{ (Z_-(y')-c_*)(Z_+(y')-c_*) \left(\varphi_*^r(y') \right)^2} \mathrm{d}y' \\
= & \wt \mu^{r}[F_*](c_*) \varphi_*^r(y) \int_{y_2}^y \frac{1}{(Z_-(y')-c_*)(Z_+(y')-c_*)\left( \varphi_*^r(y')\right)^2} \mathrm{d}y'\\
& + \varphi_*^r(y) \int_{y_2}^y\frac{\int_{y_{*}^r}^{y'} (F_* \varphi_*^r)(z) \mathrm{d}z}{ (Z_-(y')-c_*)(Z_+(y')-c_*) \left( \varphi_*^r(y') \right)^2}\mathrm{d}y' ,
\end{split}
\eeq
while the solution to Equation \eqref{eq:blan} on $[y_1, 0]$ is given by
\begin{align*}
\Phi_*^\ell (y)= & \varphi_*^\ell (y) \int_{y_1}^y\frac{\int_{y_{*}^\ell}^{y'}(F_* \varphi_*^\ell)(z) \mathrm{d}z }{(Z_-(y')-c_*)(Z_+(y')-c_*)\left( \varphi_*^\ell(y') \right)^2}\mathrm{d}y'\\
&+\wt\mu^\ell [F_*](c_*) \varphi_*^\ell(y) \int_{y_1}^y\frac{1}{(Z_-(y')-c_*)(Z_+(y')-c_*) \left( \varphi_*^\ell (y') \right)^2} \mathrm{d}y'\\
= &  {\mu}^\ell [F_*](c_*) \varphi_*^\ell(y)\int_0^y\frac{1}{(Z_-(y')-c_*)(Z_+(y')-c_*)\left( \varphi_*^\ell (y') \right)^2}\mathrm{d}y'
+\nu^\ell [F_*](c_*)\varphi_*^\ell (y)\\
& + \varphi_*^\ell (y) \int_0^y \frac{\int_{y_{*}^\ell}^{y'}(F_* \varphi_*^\ell )(z) \mathrm{d}z}
{(Z_-(y')-c_*)(Z_+(y')-c_*) \left( \varphi_*^\ell (y') \right)^2}\mathrm{d}y'.
\end{align*}
Here the coefficients $\mu^j [F_*](c_*), \wt \mu^j [F_*](c_*)$ and $\nu^j [F_*](c_*),$ $j=\ell$ or $r,$ are determined by $F_*$ and $c_*$ in a way such that $\Phi_*^r$ and $\Phi_*^\ell$ are well-defined and satisfy the conditions 
\begin{equation}\label{bdcnd}
\begin{cases}
\Phi_*^\ell (y_1) =0,  \ \Phi_*^r(y_2) =0,  \\
\Phi_*^\ell (0) = \Phi_*^r(0),  \ \partial_y \Phi_*^\ell (0) = \pa_y \Phi_*^r(0),
\end{cases}
\end{equation}
which gives us
\beqno
 \left\{\begin{array}{l}
\mu^r[F_*](c_*)=\wt{\mu}^r[F_*](c_*), \ \
 \mu^{\ell}[F_*](c_*)=\wt{\mu}^{\ell}[F_*](c_*),\\
I^r(c_*)\mu^r[F_*](c_*)+\nu^r[F_*](c_*)=-T^r[F_*](c_*),\\
I^{\ell}(c_*)\mu^{\ell}[F_*](c_*)-\nu^{\ell}[F_*](c_*)=T^{\ell}[F_*](c_*),\\
\va_*^r(0)\nu^r[F_*](c_*)-\va_*^{\ell}(0)\nu^{\ell}[F_*](c_*)=0,\\
\va_*^{\ell}(0)\mu^r[F_*](c_*)-\va_*^r(0)\mu^{\ell}[F_*](c_*)
+(Z_-(0)-c_*)(Z_+(0)-c_*)(\va_*^{\ell}\va_*^r\pa_y\va_*^r)(0)\nu^r[F_*](c_*)\\
\ \  -(Z_-(0)-c_*)(Z_+(0)-c_*)(\va^r_*\va_*^{\ell}\pa_y\va_*^{\ell})(0)\nu^{\ell}[F_*](c_*)=L[F_*](c_*).
\end{array}\right.
\eeqno
We rewrite the above set of equations in the form of matrix equation as
\begin{gather*}
\begin{bmatrix}
 I^r& 0 & 1&  0\\
 0&  I^{\ell}& 0 &  -1\\
 0& 0&  \va_*^r(0)&  -\va_*^{\ell}(0)\\
 \va_*^{\ell}(0)&  -\va^r_*(0)& \cH_0(c_*)(\va_*^{\ell}\va_*^r\pa_y\va_*^r)(0)
 & -\cH_0(c_*)(\va^r_*\va_*^{\ell}\pa_y\va_*^{\ell})(0)
 \end{bmatrix}
 \begin{bmatrix}
 \mu^r[F_*]\\
 \mu^{\ell}[F_*]\\
 \nu^r[F_*]\\
 \nu^{\ell}[F_*]
  \end{bmatrix}
=
\begin{bmatrix}
 -T^r[F_*]\\
 T^{\ell}[F_*]\\
 0\\
 L[F_*]
\end{bmatrix},
\end{gather*}
where $\cH_0(c_*):=(Z_-(0)-c_*)(Z_+(0)-c_*)$ and 
\begin{gather}
\label{Ipm}  I^\ell (c_*):= \int_{y_1}^0 \frac{1}{(Z_-(y)-c_*)(Z_+(y)-c_*) \left( \varphi_*^\ell (y) \right)^2 } \mathrm{d}y, \\
 I^r(c_*):= \int^{y_2}_0 \frac{1}{(Z_-(y)-c_*)(Z_+(y)-c_*) \left( \varphi_*^r (y) \right)^2 } \mathrm{d}y,\\
L[F_*](c_*):=\varphi_*^\ell (0)\int^{y_{*}^r}_0 (F_* \varphi_*^r)(y)\mathrm{d}y
-\varphi_*^r(0)\int^{y_{*}^\ell}_0(F_* \varphi_*^\ell)(y)\mathrm d y,\\
T^\ell [F_*](c_*):=\int_0^{y_1}\frac{\int_{y_{*}^\ell}^{y}( F_* \varphi_*^\ell)(z)\mathrm{d}z}{(Z_-(y)-c_*)(Z_+(y)-c_*) \left( \varphi_*^\ell(y) \right)^2 } \mathrm{d}y,\\
\ T^r[F_*](c_*):= \int_0^{y_2} \frac{\int_{y_{*}^r}^{y}( F_* \varphi_*^r)(z)\mathrm{d}z}{(Z_-(y)-c_*)(Z_+(y)-c_*) \left( \varphi_*^r(y)\right)^2}\mathrm{d}y.
\end{gather}
And hereafter, we denote
\begin{equation*}
\begin{split}
 \cD(c_*):= & (Z_-(0)-c_*)(Z_+(0)-c_*) \varphi_*^r(0) \varphi_*^\ell(0)\left(\va_*^{\ell}(0)\pa_y \varphi_*^r(0)  -\va_*^r(0) \pa_y\varphi_*^\ell(0) \right) I^r(c_*) I^\ell (c_*) \\
 & - \left(\varphi_*^r (0)\right)^2 I^r (c_*)-(\varphi_*^\ell(0))^2 I^\ell(c_*),
\end{split} 
\end{equation*}
which is in fact the determinant of the matrix in the matrix equation above. 

Continuing solving for the coefficients, we obtain
\begin{equation}
\begin{split}
&\mu^r[F_*](c_*) =  \wt \mu^r[F_*](c_*)\\
&:=  \frac{1}{\mathcal{D} (c_*)} (Z_-(0)-c_*)(Z_+(0)-c_*) \varphi_*^{r}(0) \varphi_*^\ell(0) \left(\partial_y\varphi_*^\ell(0) - \partial_y\varphi_*^r(0) \right) T^r[F_*](c_*) I^\ell(c_*)\\
&\quad -  \frac{1}{\mathcal{D} (c_*)}\left(\varphi_*^\ell(0) L[F_*](c_*) I^\ell(c_*) + \left(\varphi_*^r(0)\right)^2 T^r[F_*](c_*) \right) \\
& \quad-\frac{1}{\mathcal{D} (c_*)}\left(\varphi_*^r(0) \varphi_*^\ell(0) T^\ell [F_*](c_*) \right),\\
\end{split}
\end{equation}
\begin{equation}
\begin{split}
& \mu^\ell [F_*](c_*) = \wt \mu^\ell [F_*](c_*) \\
&:=  \frac{1}{\mathcal{D} (c_*)}  (Z_-(0)-c_*)(Z_+(0)-c_*) \varphi_*^r(0) \varphi_*^\ell(0) \left(\partial_y\varphi_*^\ell(0) - \partial_y\varphi_*^r(0) \right) T^\ell [F_*](c_*) I^r (c_*) \\
 &\quad +\frac{1}{\mathcal{D} (c_*)} \left( \varphi_*^r(0) L[F_*](c_*) I^r (c_*) - \left(\varphi_*^\ell(0) \right)^2 T^\ell [F_*](c_*) \right)\\
 &\quad + \frac{1}{\mathcal{D} (c_*)} \left( \varphi_*^r(0) \varphi_*^\ell (0)T^r [F_*](c_*) \right) ,
 \end{split}
\end{equation}
\begin{equation}
\begin{split}
\nu^r [F_*](c_*) =: & \frac{1}{\mathcal D(c_*)} \left(  \varphi_*^\ell(0) L[F_*](c_*) I^r(c_*) I^\ell(c_*) -\varphi_*^r(0) \varphi_*^\ell(0) T^\ell [F_*](c_*) I^r(c_*) \right)\\
& +\frac{1}{\mathcal D(c_*)} \left(  \left( \varphi_*^\ell (0)\right)^2 T^r [F_*](c_*) I^\ell (c_*) \right),\\
\end{split}
\end{equation}
\begin{equation}
\begin{split}
\nu^\ell [F_*](c_*):= & \frac{1}{\mathcal D(c_*)} \left( \varphi_*^r(0) L[F_*](c_*) I^r (c_*) I^\ell (c_*) +\varphi_*^r (0)\varphi_*^\ell(0) T^r [F_*](c_*) I^\ell (c_*) \right) \\
& + \frac{1}{\mathcal D(c_*)} \left( \left( \varphi_*^r(0) \right)^2 T^\ell [F_*](c_*) I^r (c_*) \right) .
\end{split}
\end{equation}

Thus, we can verify that
\beno
\Phi_*(y,c_*)=\begin{cases}
\Phi_*^\ell(y,c_*) \text{ on }[y_1,  0],\\
\Phi_*^r(y, c_*) \text{ on }[0, y_2]
\end{cases}
\eeno
is well-defined and is the unique $C^1$- solution to \eqref{eq:blan} on $\mathcal I_0.$

To facilitate the estimation of $I^k(c_*)$ for $k=r$ or $\ell$, we introduce also the following quantities
\begin{equation}\notag
\begin{split}
& I_1^r(c_*):=\int_0^{y_2}\f{1}{(Z_-(y)-c_*)(Z_+(y)-c_*)}\left(\f{1}{ \left( \varphi_*^r(y) \right)^2}-1\right) \mathrm{d}y+\int_0^{y_2}\f{1}{2b(y)(Z_-(y)-c_*)}\mathrm{d}y, \\
& I_1^\ell(c_*):=\int_{y_1}^0\f{1}{(Z_-(y)-c_*)(Z_+(y)-c_*)}\left(\f{1}{ \left(\varphi_*^\ell(y) \right)^2}-1\right)\mathrm{d}y+\int_{y_1}^0\f{1}{2b(y)(Z_-(y)-c_*)}\mathrm{d}y,
\end{split}
\end{equation}
\begin{equation}\notag
\begin{split}
I_2^r(  \sigma(c_*) ):=&-\f{1}{2b(y_2)V'(y_2)}\ln \left( \f{|V(y_2)-  {\sigma(c_*) }|}{|V(y_2)+  {\sigma(c_*) }|} \right)+ \f{i}{2b(y_2)V'(y_2)}\arctan\left( \f{\text{Im}\,  {\sigma(c_*) }}{|V(y_2)- \text{Re}\,  {\sigma(c_*) }|} \right)\\
& +\int_0^{y_2} \pa_y\left(\f{1}{2b(y)V'(y)}\right) \text{Log}\left( \f{V(y)-  {\sigma(c_*) }}{V(y)+  {\sigma(c_*) }} \right) \mathrm{d}y\\
&+\f{i}{2b(y_2)V'(y_2)}\arctan\left( \f{\text{Im}\,  {\sigma(c_*) }}{|V(y_2)+\text{Re}\,  {\sigma(c_*) }|} \right) ,
\end{split}
\end{equation}
\begin{equation}\notag
\begin{split}
I_2^\ell(  {\sigma(c_*) }):=& \f{1}{2b(y_1)V'(y_1)}\ln \left( \f{|V(y_1)-  {\sigma(c_*) }|}{|V(y_1)+  {\sigma(c_*) }|} \right)+\int_{y_1}^0\pa_y\left(\f{1}{2b(y)V'(y)}\right)\text{Log} \left( \f{V(y)-  {\sigma(c_*) }}{V(y)+  {\sigma(c_*) }}\right) \mathrm{d}y\\
&+\f{i}{2b(y_1)V'(y_1)}\arctan\left(\f{\text{Im}\,  {\sigma(c_*) }}{|V(y_1)+ \text{Re}\,  {\sigma(c_*) }|}\right)\\
&+\f{i}{2b(y_1)V'(y_1)}\arctan \left(\f{\text{Im}\,  {\sigma(c_*) }}{|V(y_1)-\text{Re}\,  {\sigma(c_*) }|} \right)-\f{2i\pi}{2b(y_1)V'(y_1)}+\f{2i\pi}{2b(0)V'(0)},
\end{split}
\end{equation}
where $\text{Log}$ is the complex logarithm with the principal value of the argument in $(-\pi, \pi].$ 

We shall prove the following auxiliary estimates on $I^k(\sigma(c_*) ),$ $k= \ell$ or $r$ and $\cD(c_* )$ which will help us characterize the behaviors of the coefficients $\mu^k[F_*](c_*), \wt \mu^k[F_*](c_*)$ and $\nu^k[F_*](c_*),$ $k= \ell$ or $r,$ in the solution formulae to Equation \eqref{eq:blan}.
\begin{lemma}\label{38} 
Assume that $\mathrm{Im}\,c_* >0$ and $\mathrm{Im}\,\s(c_*)>0.$ It holds that   
\begin{equation}\label{381}
2{\sigma(c_*) } I^k(c_*)=-\f{i\pi}{b(0)\sqrt{2Z''_+(0)}}+2{\sigma(c_*) }I_1^{k}(c_*)+I_2^{k}({\sigma(c_*) }), \ k= \ell \text{ or }r.
\end{equation}
Moreover, there exist some $\delta_0>0$ and a constant $C$ depending only on $\alpha,$ such that if $|{\sigma(c_*) }|<\d_0,$ then the following estimates are true for $k=\ell$ or $r$ 
\begin{eqnarray}
\label{ipmg}& |I_1^{k}(c_*)|\leq C,  \\
\label{i2pm}& |I_2^{k}({\sigma(c_*) })|\leq C|{\sigma(c_*) }|^{\f14}, \\
\label{384} & C^{-1}\leq |2{\sigma(c_*) }I^{k}(c_*)|\leq C.
\end{eqnarray}
In particular,  
$$\lim_{c_*\to Z_+(0)} 2{\sigma(c_*) }I^{k}(c_*) = -\frac{i\pi}{b(0)\sqrt{2Z''_+(0)}}, \ k= \ell \text{ or } r.$$
\end{lemma}

\begin{proof}
Splitting the integral $I^r(c_*)$ and utilizing the function $V,$ we have
\begin{align*}
2{\sigma(c_*) } I^r(c_*) =& 2{\sigma(c_*) }\int_0^{y_2} \f{1}{(Z_-(y)-c_*)(Z_+(y)-c_*)\left( \varphi_*^r(y) \right)^2} \mathrm{d}y\\
= & 2{\sigma(c_*) }\int_0^{y_2}\f{1}{(Z_-(y)-c_*)(Z_+(y)-c_*)}\left(\f{1}{ \left( \varphi_*^r(y) \right)^2}-1\right) \mathrm{d}y\\
&+2{\sigma(c_*) } \int_0^{y_2}\f{1}{2b(y)(Z_-(y)-c_*)}\mathrm{d}y
-2{\sigma(c_*) } \int_0^{y_2}\f{1}{2b(y)(Z_+(y)-c_*)}\mathrm{d}y\\
= & 2{\sigma(c_*) } I_1^r(c_*)-\int_0^{y_2}\f{1}{ 2b(y)(V(y)-{\sigma(c_*) }) } \mathrm{d}y +\int_0^{y_2} \f{1}{ 2b(y)(V(y)+{\sigma(c_*)) }  }\mathrm{d}y \\
= & 2{\sigma(c_*)} I_1^r(c_*)-\int_0^{y_2} \f{1}{2b(y)V'(y)}\pa_y\left(\text{Log}(V(y)-{\sigma(c_*) })\right)\mathrm{d}y \\
& +\int_0^{y_2} \f{1}{2b(y)V'(y)}\pa_y\left(\text{Log}(V(y)+{\sigma(c_*) })\right)\mathrm{d}y.
\end{align*}
Integration by parts yields
\begin{align*}
2{\sigma(c_*) } I^r(c_*) = & 2{\sigma(c_*) } I_1^r(c_*)-\f{\text{Log}(V(y)-{\sigma(c_*) })}{2b(y)V'(y)}\Big|_{y=0}^{y=y_2}+\f{\text{Log}(V(y)+{\sigma(c_*) })}{2b(y)V'(y)}\Big|_{y=0}^{y=y_2} \\
& +\int_0^{y_2} \pa_y \left( \f{1}{2b(y)V'(y)}\right) \left( \text{Log}\left( V(y)-{\sigma(c_*) } \right) -\text{Log}\left(V(y)+{\sigma(c_*) } \right) \right)\mathrm{d}y\\
=& 2{\sigma(c_*)}I_1^r(c_*)-\f{1}{2b(y_2)V'(y_2)}\ln \left( \f{|V(y_2)-{\sigma(c_*) }|}{|V(y_2)+{\sigma(c_*) }|}\right)\\
&+\int_0^{y_2} \pa_y\left(\f{1}{2b(y)V'(y)}\right)\text{Log}\left( \f{V(y)-{\sigma(c_*) }}{V(y)+{\sigma(c_*) }} \right) \mathrm{d}y\\
&+\f{i}{2b(y_2)V'(y_2)}\arctan \left( \f{\text{Im}\,{\sigma(c_*) }}{|V(y_2)+ \text{Re}\,{\sigma(c_*) }|} \right)\\
&+\f{i}{2b(y_2)V'(y_2)}\arctan\left( \f{\text{Im}\,{\sigma(c_*) }}{|V(y_2)+ \text{Re}\,{\sigma(c_*) }|} \right)  -\f{i\pi}{2b(0)V'(0)}\\
=& 2{\sigma(c_*) }I_1^r(c_*)-\f{i\pi}{2b(0)V'(0)}+I_2^r({\sigma(c_*) }).
\end{align*}

Similarly, we have
\begin{align*}
2{\sigma(c_*) }I^\ell (c_*) = & 2{\sigma(c_*) }I_1^\ell (c_*)-\int_{y_1}^0\f{1}{2b(y)(V(y)-{\sigma(c_*) })}\mathrm{d}y+\int_{y_1}^0\f{1}{2b(y)(V(y)+{\sigma(c_*) })}\mathrm{d}y\\
= & 2{\sigma(c_*) }I_1^\ell(c_*)-\f{\text{Log}(V(y)-{\sigma(c_*) })}{2b(y)V'(y)}\Big|_{y=y_1}^{y=0}\\
&+\int_{y_1}^0\pa_y\left(\f{1}{2b(y)V'(y)}\right)\text{Log}(V(y)-{\sigma(c_*) })\mathrm{d}y +\f{\text{Log}(V(y)+{\sigma(c_*) })}{2b(y)V'(y)}\Big|_{y=y_1}^{y=0}\\
&-\int_{y_1}^0\pa_y\left(\f{1}{2b(y)V'(y)}\right)\text{Log}(V(y)+{\sigma(c_*) })\mathrm{d}y\\
= & 2{\sigma(c_*) }I_1^\ell(c_*)+\f{1}{2b(y_1)V'(y_1)}\ln \left( \f{|V(y_1)-{\sigma(c_*) }|}{|V(y_1)+{\sigma(c_*) }|} \right)\\
&+\int_{y_1}^0\pa_y\left(\f{1}{2b(y)V'(y)}\right)\text{Log}\left( \f{V(y)-{\sigma(c_*) }}{V(y)+{\sigma(c_*) }} \right) \mathrm{d}y\\
&-\f{i}{2b(y_1)V'(y_1)}\left(2\pi-\arctan\f{\text{Im}\,{\sigma(c_*) }}{|V(y_1)+\text{Re}\,{\sigma(c_*) }|}-\arctan\f{\text{Im} {\sigma(c_*) }}{|V(y_1)-\text{Re}\,{\sigma(c_*) }|}\right)\\
& -\f{i}{2b(0)V'(0)}\left(-\pi+\arctan\f{ \text{Im}\,{\sigma(c_*) }}{| \text{Re}\,{\sigma(c_*) }|}-\arctan\f{ \text{Im}\,{\sigma(c_*) }}{| \text{Re}\,{\sigma(c_*) }|}\right)\\
= & 2{\sigma(c_*) }I_1^\ell(c_*)-\f{i\pi}{2b(0)V'(0)}+I_2^\ell ({\sigma(c_*) }).
\end{align*}
We have thus shown \eqref{381}.

We have the bound on $I^k_1(c_*),$ $k=\ell$ or $r,$ in \eqref{ipmg} for $|{\sigma(c_*) }|<\d_0 \ll 1$ by the fact that $\left|\varphi_*^{k} \right| \geq \frac{1}{2}$ and $\left|\varphi_*^{k}(y)-1 \right|\leq C\left|y-y_{*}^{k} \right|^2,$ $k=\ell$ or $r.$

We proceed to prove \eqref{i2pm}. As $ \left|V'(y_2) \right| >0,$ by taking $|\text{Im}\,({\sigma(c_*) })|\ll \left|\text{Re}({\sigma(c_*) })\right|<\d_0,$ we obtain the following 
\begin{align}\label{estarg}
\left|\f{i}{2b(y_2)V'(y_2)}\left(\arctan\f{\text{Im}\,{\sigma(c_*) }}{|V(y_2)-\text{Re}\,{\sigma(c_*) }|}+\arctan\f{\text{Im}\,{\sigma(c_*) }}{|V(y_2)+\text{Re}\,{\sigma(c_*) }|}\right)\right|\leq C|{\sigma(c_*) } |.
\end{align}
For sufficiently small $\d_0$ and $|\sigma(c_*) |<\d_0,$ we have
\begin{equation}\label{estlog}
\f{1}{2b(y_2)V'(y_2)}\ln\left( \f{|V(y_2)-{\sigma(c_*) }|}{|V(y_2)+{\sigma(c_*) }|} \right) 
\leq  C|{\sigma(c_*) }|,
\end{equation}
as $$\f{1}{2b(y_2)V'(y_2)}\ln\left( \f{|V(y_2)-{\sigma(c_*) }|}{|V(y_2)+{\sigma(c_*) }|} \right) = \f{1}{2b(y_2)V'(y_2)}\ln\left(1-\f{2V(y_2) \text{Re}\, {\sigma(c_*) }}{|V(y_2)+\text{Re}\,{\sigma(c_*) }|^2+|\text{Im}\,({\sigma(c_*) })|^2}\right).$$

We split the integral as follows --
\begin{align*}
&\int_0^{y_2} \pa_y\left(\f{1}{2b(y)V'(y)}\right) \text{Log} \left( \f{V(y)-{\sigma(c_*) }}{V(y)+{\sigma(c_*) }} \right)\mathrm{d}y \\
&= \int_{E}\pa_y\Big(\f{1}{2b(y)V'(y)}\Big) \text{Log} \left( \f{V(y)-{\sigma(c_*) }}{V(y)+{\sigma(c_*) }} \right)\mathrm{d}y\\
&\quad +\int_{E^c}\pa_y\Big(\f{1}{2b(y)V'(y)}\Big) \text{Log} \left( \f{V(y)-{\sigma(c_*) }}{V(y)+{\sigma(c_*) }} \right) \mathrm{d}y\\
&:=K_1({\sigma(c_*) })+K_2({\sigma(c_*) }),
\end{align*}
where the set $E$ is defined as $E := \left \{y \in [0, y_2] : V(y)< M\sqrt{|{\sigma(c_*) }|} \right\}$ for sufficiently large $M$ independent of ${\sigma(c_*) },$ while $E^c:= \left \{y \in [0, y_2]:V(y)\geq M\sqrt{|{\sigma(c_*) }|} \right\}.$

For $|{\sigma(c_*) }|<\d_0$ with sufficiently small $\d_0,$ we have the estimate
\begin{equation}\label{estk1}
|K_1({\sigma(c_*) })|\leq C\sqrt[4]{|{\sigma(c_*) }|}+C\sqrt{|{\sigma(c_*) }|}\leq C\sqrt[4]{|{\sigma(c_*) }|},
\end{equation}
as $\text{Log}\left( \f{V(y)-{\sigma(c_*) }}{V(y)+{\sigma(c_*) }} \right) \in L^2$ in the set $E,$ in particular,
$$\text{Log}\left( \f{V(y)-{\sigma(c_*) }}{V(y)+{\sigma(c_*) }} \right) = \ln \left( \f{|V(y)-{\sigma(c_*) }|}{|V(y)+{\sigma(c_*) }|} \right) +i\arg(V(y)-{\sigma(c_*) })-i\arg(V(y)+{\sigma(c_*) }) .$$

Meanwhile, in the set $E^c$, it holds that
\begin{align*}
\text{Log}\left( \f{V(y)-{\sigma(c_*) }}{V(y)+{\sigma(c_*) }}\right)=& \ln \left( \frac{|V(y)-{\sigma(c_*) }|}{ |V(y)+{\sigma(c_*) }|} \right)-i\arctan \left(\f{\text{Im}\,{\sigma(c_*) }}{|V(y)-\text{Re}\,{\sigma(c_*) }|} \right) \\
&-i\arctan \left(\f{ \text{Im}\,{\sigma(c_*) }}{|V(y)+ \text{Re}\,{\sigma(c_*)}|}\right),
\end{align*}
from which we infer that 
\begin{equation}\label{estk2}
|K_2({\sigma(c_*) })| \leq C\sqrt{|{\sigma(c_*) }|},
\end{equation}
when $|{\sigma(c_*) }|<\d_0$ for small enough $\d_0.$ 

Combining \eqref{estarg}, \eqref{estlog}, \eqref{estk1} and \eqref{estk2}, we have
$$ |I_2^r({\sigma(c_*) })|\leq C \left( \sqrt[4]{|{\sigma(c_*) }|}+|{\sigma(c_*) }| + |\sigma(c_*)| \right) \leq C\sqrt[4]{|{\sigma(c_*) }|}.$$

We note that the terms $\f{1}{2b(y_1)V'(y_1)}\ln \left(\f{|V(y_1)-{\sigma(c_*) }|}{|V(y_1)+{\sigma(c_*) }|} \right)$ and $\f{i}{2b(y_1)V'(y_1)}\arctan\left( \f{\text{Im}\,({\sigma(c_*) })}{|V(y_1)\pm \text{Re}\,({\sigma(c_*) })|}\right)$ enjoy estimates similar to \eqref{estlog} and \eqref{estarg}, respectively.

Introducing the sets 
$$\mathcal E := \left \{y \in [y_1, 0] : V(y)< M\sqrt{|{\sigma(c_*) }|} \right\} \text{ and }\mathcal E^c:= \left \{y \in [y_1, 0]:V(y)\geq M\sqrt{|{\sigma(c_*) }|} \right\},$$ where $M$ is supposed to be large and independent of $\sigma(c_*),$ we split the integral in $I^{\ell}_2(\sigma(c_*):$
\begin{equation}\notag
\int_{y_1}^0\pa_y\left(\f{1}{2b(y)V'(y)}\right)\text{Log}\left( \f{V(y)-{\sigma(c_*) }}{V(y)+{\sigma(c_*) }}\right) \mathrm{d}y-\f{2i\pi}{2b(y_1)V'(y_1)}+\f{2i\pi}{2b(0)V'(0)} =: \mathcal K_1(\sigma(c_*)) + \mathcal K_2(\sigma(c_*)),
\end{equation}
where
\begin{gather*}
\mathcal K_1(\sigma(c_*)):= \int_{\mathcal E} \pa_y\left(\f{1}{2b(y)V'(y)}\right)\text{Log}\left( \f{V(y)-{\sigma(c_*) }}{V(y)+{\sigma(c_*) }}\right) \mathrm{d}y, \\
\mathcal K_2(\sigma(c_*)):=\int_{\mathcal E^c} \pa_y\left(\f{1}{2b(y)V'(y)}\right)\text{Log}\left( \f{V(y)-{\sigma(c_*) }}{V(y)+{\sigma(c_*) }}\right) \mathrm{d}y-\f{2i\pi}{2b(y_1)V'(y_1)}+\f{2i\pi}{2b(0)V'(0)}.
\end{gather*}

On the set $\mathcal E$, we have $\text{Log} \left( \f{V(y)-{\sigma(c_*) }}{V(y)+{\sigma(c_*) }}\right) \in L^2,$ as 
\beno
\text{Log} \left( \f{V(y)-{\sigma(c_*) }}{V(y)+{\sigma(c_*) }}\right)=\ln \left( \f{|V(y)-{\sigma(c_*) }|}{|V(y)+{\sigma(c_*) }|}\right)+i\arg(V(y)-{\sigma(c_*) })-i\arg(V(y)+{\sigma(c_*) }).
\eeno
For $\d_0$ small enough, it holds that
\begin{align}\notag
|\mathcal K_1({\sigma(c_*) })|\leq C\sqrt[4]{|{\sigma(c_*) }|}+C\sqrt{|{\sigma(c_*) }|}\leq C\sqrt[4]{|{\sigma(c_*) }|}.
\end{align}

On the $\mathcal{E}^c,$ we have
\begin{align*}
\text{Log} \left( \frac{V(y)-{\sigma(c_*) }} {V(y)+{\sigma(c_*) }} \right)= & \ln \left( \frac{ |V(y)-{\sigma(c_*) }|} {|V(y)+{\sigma(c_*) }|} \right) -i\arctan \left( \f{\text{Im}\,({\sigma(c_*) })}{|V(y)-\text{Re}\,({\sigma(c_*) })|} \right) \\
& -i\arctan \left( \f{\text{Im}\,({\sigma(c_*) })}{|V(y)+\text{Re}\,({\sigma(c_*) })|} \right)-2i\pi,
\end{align*}
from which we can deduce that for $\d_0$ small enough
\beno
|\mathcal K_2({\sigma(c_*) })|\leq C\sqrt{|{\sigma(c_*) }|}.
\eeno
It is then evident that
\beno
|I_2^\ell ({\sigma(c_*) }) \leq C\sqrt[4]{|{\sigma(c_*) }|}.
\eeno

By \eqref{ipmg} and \eqref{i2pm}, we have \eqref{384}.

\end{proof}

From Lemma \ref{38}, we can deduce the following corollary.
\begin{corol}\label{311}
It holds that for $|\sigma(c_*)|<\d_0$ with $\d_0>0$ small enough,
\begin{align}
\f{C^{-1}}{|\sigma(c_*)|}\leq|\cD(c_*)|\leq \f{C}{|\sigma(c_*)|}.
\end{align}
In addition, if $\mathrm{Im}\,c_* >0,$ then
$$\lim_{c_*\to Z_+(0)} {\sigma(c_*)} \cD(c_*) = \frac{i\pi}{b(0)\sqrt{2Z''_+(0)}}.$$
\end{corol}

Recalling \eqref{eq: Theta_0+} and its counterpart on $[y_1, 0],$ we are ready to prove Lemma \ref{lem: phi*}

\noindent \textit{Proof of Lemma \ref{lem: phi*}. } 
It is easy to check that 
\begin{align*}
\left |L[F_{*}](c_*) \right|\leq C \left( |y_*^r|+|y_*^l| \right) \|F_* \|_{L^\infty}.
\end{align*}
We also have the following estimate for $T^\ell [F_*](c_*),$ 
\begin{align*}
\left |T^\ell [F_*](c_*) \right| \leq & \left|\int_0^{y_1}\f{\int_{y_*^\ell}^yF_*(z)\,\mathrm{d}z}{(Z_-(y)-c_*)(Z_+(y)-c_*)}\mathrm{d}y\right|\\
&+\left|\int_0^{y_1}\int_{y_*^\ell}^y \left( \f{F_*(z)}{(Z_-(y)-c_*)(Z_+(y)-c_*)} \right) \left(\f{\varphi_*^\ell(z)}{\varphi_*^\ell(y)^2}-1\right)\mathrm{d}z\,\mathrm{d}y\right|\\
\leq & C \|F_*\|_{L^{\infty}}\left(\int_0^{y_1}\f{1}{|y|+|y_*^\ell|} \mathrm{d} y +1 \right) \\
\leq & C \left(1+ \left |\ln \left( |y_*^\ell |\right) \right| \right)\|F_*\|_{L^{\infty}}.
\end{align*}
Similarly, 
\beno
|T^r[F_*](c_*)| \leq C \left(1+ \left|\ln ( |y_*^r| ) \right| \right)\|F_*\|_{L^{\infty}}.
\eeno

By Lemma \ref{38}, Corollary \ref{311} and the fact that $|y_*^r| \sim |y_*^\ell|,$ we have the following bounds on the coefficients $\mu^j[F_*](c_*), \wt \mu^j[F_*](c_*)$ and $\nu^j[F_*](c_*),$ $j=\ell$ or $r,$
\begin{eqnarray}
& \label{mue} \left| \mu^j[F_*](c_*) \right| \leq C |\s(c_*)|\left(1+\left|\ln (|y_*^r|) \right| \right)\|F_*\|_{L^{\infty}},\\
& \label{wmue}\left| \wt \mu^j[F_*](c_*) \right| \leq C |\s(c_*)|\left(1+\left|\ln (|y_*^r|) \right| \right)\|F_*\|_{L^{\infty}},\\
& \label{nue} \left| \nu^j [F_*](c_*) \right| \leq C \left(1+\left|\ln (|y_*^r|) \right| \right)\|F_*\|_{L^{\infty}}.
\end{eqnarray}

We recall the explicit formula of $\Phi_*^r$ given by \eqref{eq: Theta_0+} 
\beq\ \notag
\begin{split}
\Phi_*^r(y)= & \nu^r[F_*](c_*) \varphi_*^r(y) + {\mu}^{r}[F_*](c_*)  \varphi_*^r(y) \int_0^y \frac{1}{ (Z_-(y')-c_*)(Z_+(y')-c_*)\left( \varphi_*^r(y') \right)^2} \mathrm{d}y'\\
&+ \varphi_*^r(y) \int_0^y \frac{\int_{y_{*}^r}^{y'}(F_* \varphi_*^r)(z)\mathrm{d}z}
{ (Z_-(y')-c_*)(Z_+(y')-c_*)\left(\varphi_*^r(y') \right)^2} \mathrm{d}y' =: J_1 + J_2 +J_3, \ \text{for } y \in (0, y_*^r);
\end{split}
\eeq
and
\beq\ \notag
\begin{split}
\Phi_*^r(y)= & \wt{\mu}^{r}[F_*](c_*)  \varphi_*^r(y) \int^y_{y_2} \frac{1}{(Z_-(y')-c_*)(Z_+(y')-c_*)\left( \varphi_*^r(y') \right)^2} \mathrm{d}y'\\
&+ \varphi_*^r(y) \int^{y}_{y_2} \frac{\int_{y_{*}^r}^{y'}(F_* \varphi_*^r)(z)\mathrm{d}z}
{ (Z_-(y')-c_*)(Z_+(y')-c_*) \left(\varphi_*^r(y') \right)^2} \mathrm{d}y' =: \wt J_1 + \wt J_2, \ \text{for } y \in (y_*^r, y_2).
\end{split}
\eeq

From \eqref{nue}, it is clear that for $y \in (0, y_*^r)$, 
\beno
|J_1| \leq C\left(1+\left|  \ln\left( |y_*^r | \right) \right| \right) \|F_* \|_{L^\infty}\leq C\left(1+\left|  \ln\left( |y | \right) \right| \right) \|F_* \|_{L^\infty}.
\eeno

From \eqref{mue} and the fact that $\left(1+\left|  \ln\left( |y_*^r | \right)\right|\right) \lesssim |\s(c_*)|^{-\gamma},$ $0 < \gamma < \frac{1}{4},$ for $c_*$ close to $Z_+(0),$  we infer that
\begin{equation}\label{esj2}
\begin{split}
\left| J_2 \right| \leq & C \|F_* \|_{L^\infty}  |\s(c_*)|^{1-\gamma} \left| \int_0^y\f{1}{(|y'|+|\s(c_*)|)(y'-y_*^r)}\mathrm{d}y' \right|\\
\leq & C \|F_* \|_{L^\infty} |\sigma(c_*)|^\gamma \left|\int_0^y\f{1}{|y'-y_*^r|^{1+2\gamma}} \mathrm{d}y' \right|\\
\leq & C|\sigma(c_*)|^\gamma \|F_* \|_{L^\infty}  \left( \f{1}{|y-y_*^r|^{2\gamma}}+1 \right),
\end{split}
\end{equation}
which shows that $\|J_2\|_{L^2} \leq C \|F_*\|_{L^\infty}.$ In a similar way, we can show that $ \|\wt J_1  \|_{L^2} \leq C \|F_*\|_{L^\infty}.$

The term $J_3$ enjoys the following estimate
\begin{equation}\notag
\left| \int_0^y \frac{\int_{y_{*}^r}^{y'}(F_* \varphi_*^r)(z)\mathrm{d}z}
{ (Z_-(y')-c_*)(Z_+(y')-c_*)\left(\varphi_*^r(y') \right)^2} \mathrm{d}y' \right| \leq \|F_*\|_{L^\infty} \left( 1+ \left| \ln (|y_*^r |+|y|) \right|\right),
\end{equation}
and the estimate for $\wt J_2$ is similar.

Therefore, it holds that 
\begin{align*}
\| \Phi_*^r \|_{L^2(0,y_2)} &\leq \|J_1\|_{L^2(0,y_*^r)}+\|J_2\|_{L^2(0,y_*^r)}+\|J_3\|_{L^2(0,y_*^r)}+\|\wt J_1\|_{L^2(y_*^r,y_2)}+\|\wt J_2\|_{L^2(y_*^r,y_2)}\\
&\leq C \|F_* \|_{L^\infty}.
\end{align*}
Similarly, we have $\| \Phi_*^\ell \|_{L^2(y_1,0)} \leq C \|F_* \|_{L^\infty},$ which concludes our proof. 
\qed
\begin{remark}
In the case that $c_*\in\mathrm{Ran}\,Z_-$ with $\mathrm{Im}\,c_* >0$, without loss of generality, we may assume that $Z'_-(0)=0$ and $Z''_-(0)>0.$ We can instead let $V(y)^2=Z_-(y)-Z_-(0)$ and $(\s(c_*))^2=c_*-Z_-(0).$ Then by the same argument as in the proof of Lemma \ref{38}, we obtain
$$\lim_{c_*\to Z_-(0)}2\s(c_*)I^k(c_*)=\f{i\pi}{b(0)\sqrt{2Z''_-(0)}},\  k= \ell \text{ or } r$$
and 
$$
\lim_{c_*\to Z_-(0)}\s(c_*)\cD(c_*)=-\f{i\pi}{b(0)\sqrt{2Z''_-(0)}}.
$$
\end{remark}

\noindent \textit{Proof of Lemma \ref{stgcv}.} Recalling Lemma \ref{lem: m}, which asserts the strong convergence  
$$\Phi_n \rightarrow \Phi \text{ in }L^2 \text{ and }(Z_--c_n)(Z_+-c_n)\pa_y \Phi_n \rightarrow (Z_--c)(Z_+-c)\pa_y \Phi \text{ in }H^1,$$
away from the critical points in $Z_+^{-1}(c),$ along with Equation \eqref{eq:2}, we know that
\begin{equation}
\begin{split}
F^*_n \rightarrow F^*_\infty:= & F(y) -\frac{\Phi(y_2) -\Phi(y_1)}{(y_2 - y_1)} Z_+'(y)(Z_-(y) - c) \\
& - \frac{\Phi(y_2) -\Phi(y_1)}{(y_2 - y_1)} Z_-'(y)(Z_+(y) - c) + \al^2 (Z_+ -c)(Z_- -c) \mathcal L_{\Phi} (y) 
\end{split}
\end{equation}
in $H^1$ from our assumption that $F_n \rightarrow F$ in $H^1$ as $n \to \infty.$ 

From Lemma \ref{38} and Corollary \ref{311}, we know that as $c_n \to c$ and $F^*_n \rightarrow F^*_\infty $ in $H^1,$ 
\begin{equation}\notag
\wt{\mu}^{r}[F^*_n](c_n)  \varphi_*^r(y) \int^y_{y_2} \frac{1}{(Z_--c_n)(Z_+-c_n)\left( \varphi_*^r(y') \right)^2} \mathrm{d}y' \rightarrow 0 \text{ in }L^2(0, y_2),
\end{equation}
whose analogues also hold for the terms corresponding to $\wt{\mu}^\ell[F^*_n](c_n),$ ${\mu}^\ell[F^*_n](c_n)$ and ${\mu}^r[F^*_n](c_n).$

Therefore, we know from \eqref{eq: Theta_0+} that as $c_n \to c=Z_+(0)$ and $F^*_n \rightarrow F^*_\infty $ in $H^1,$ 
\beq\ \notag
\Phi^*_n \rightarrow \Phi_\infty^* \text{ in } L^2(y_1, y_2), \text{ with }\Phi_\infty^*(y) := \begin{cases} \varphi^\ell(y) \int^{y}_{y_1} \frac{\int_{0}^{y'}(F^*_\infty \varphi^\ell)(z)\mathrm{d}z}
{ (Z_--c)(Z_+-c)\left(\varphi^\ell (y') \right)^2} \mathrm{d}y'  \ \text{on } [y_1, 0),\\ 
\varphi^r(y) \int^{y}_{y_2} \frac{\int_{0}^{y'}(F^*_\infty \varphi^r)(z)\mathrm{d}z}
{(Z_--c)(Z_+-c)\left(\varphi^r(y') \right)^2} \mathrm{d}y'  \ \text{on } (0, y_2],
\end{cases}
\eeq
where $\varphi^\ell$ and $\varphi^r$ are the solutions to the corresponding homogeneous equation, as constructed in Appendix \ref{sturm}. It is then evident that $\Phi_n \rightarrow \Phi$ in $L^2(\mathcal I_0).$

It follows from \eqref{eq: Theta_0+} and \eqref{nue} that
\begin{align*}
|\Phi^*_n (0)| \leq & \left|\nu^r[F^*_n](c_n) \right| \leq  C\left(1+ \left|\ln\left( \sqrt{ |(c_n - Z_+(0))(c_n -Z_-(0))|} \right)\right| \right)\|F_n^* \|_{L^\infty}\\
\leq  & C|(Z_+(0) -c_n)(Z_-(0) -c_n) |^{-\frac{1}{4}}.
\end{align*}

Differentiating \eqref{eq: Theta_0+} with $c_* =c_n$ at $y=0$  yields
\begin{align*}
&\left| \partial_y \Phi^*_n (0) \right| = \left| \partial_y \Phi_*^r (0,c_n) \right| \\
&\leq  \left| \nu^r[F_n^*](c_n) \partial_y \varphi_*^r(0) \right| + \left| \frac{1}{(Z_-(0)-c_n)(Z_+(0)-c_n)\varphi_*^r(0)  } \left( {\mu}^{r}[F_n^*](c_n) + \int_{y_{*}^r}^{0}(F_n^* \varphi_*^r)(z)\mathrm{d}z \right) \right| \\
&\leq  C \left( 1+ \left| \ln (|y_*^r|) \right|  \right) \left( \frac{1}{\left| (Z_-(0)-c_n)(Z_+(0)-c_n) \right|^\frac{1}{2} } +1 \right) \\
&\leq C \left|(Z_-(0)-c_n)(Z_+(0)-c_n) \right|^{- \frac{3}{4} }.
\end{align*}

We note that $|(Z_--c_n)(Z_+-c_n)|\geq C^{-1}>0$ at $y_1$ and $y_2$ as the two points are far away enough from $(Z_+)^{-1}(c).$ By the facts
\beno
(Z_--c_n)(Z_+-c_n)\pa_y\Phi_n\in L^{\infty} \text{ and } \Phi_n\in L^2,
\eeno
we know that $|\Phi_n(y_1)|+|\Phi_n(y_2)|+|\pa_y\Phi_n(y_1)|+|\pa_y\Phi_n(y_2)|\leq C.$ The estimates on $\Phi^*_n(y_0)$ and $\pa_y \Phi^*_n(y_0)$ then imply \eqref{f0} and \eqref{pf0}.
\qed

\subsection{Proofs of Proposition \ref{prop: lap} and Proposition \ref{prop: conti}}\label{proof51} 

As in previous proofs, we restrict ourselves to the case $c \in \mathrm{Ran}\,Z_+.$ By our assumptions on $u$ and $b,$ given $c \in \text{Ran}\, Z_+,$ we can assume that 
$$(Z_+)^{-1}(c) = \left \{ y_{c, 1}, y_{c, 2}, ..., y_{c, k} ; y_{0,1}, y_{0, 2},..., y_{0,m} \right \},$$
where $y_{c, i},$ $i=1,2,...,k $ are the points at which $Z_+$ is monotone, i.e., $|Z'_+(y_{c, i})| > 0,$ whereas $y_{0,i},$ $i=1,2,...,m$ are the critical points, where $Z'_+(y_{0,i})=0$ (and $|Z_+''(y_{0,i})| \not =0$). 

For each $y_{c, i},$ there exists an interval $\mathcal I_i$ such that $y_{c,i} \in \mathcal I_i$ and $Z_+'(y_{c,i}) Z_+'(y) >0, \forall y \in \mathcal I_i,$ whereas for each critical point $y_{0, j}$ we may find an interval $\mathcal I_{0,j}$ containing $y_{0,j}$ such that $Z_+''(y_{0,j}) Z_+''(y) >0, \forall y \in\mathcal I_{0,j}.$ The rest of $\mathbb T,$ consisting of regions far away from the set $(Z_+)^{-1}(c),$ can also be covered by finitely many intervals, which we denote as $\{ \mathcal I_{a, i}\}_{i=1}^{\wt n}.$   The intervals $\mathcal I_{i},$ $i=1,2,...,k,$ $\mathcal I_{0, i},$ $i=1,2,...,m,$ and $\mathcal I_{a,i},$ $i=1,2,..., \wt n,$ can be chosen in the way such that 
$$[-\pi, \pi] = \left( \cup_{i=1}^k \mathcal I_i \right) \cup \left(\cup_{i=1}^m \mathcal I_{0,i} \right) \cup \left( \cup_{i=1}^{\wt n} \mathcal I_{a, i} \right),$$ 
while each of the intervals overlaps only the ones next to it, with the size of the overlap not exceeding $\frac{1}{10}\min_{1 \leq i \leq k} |\mathcal I_i|$ and  $\frac{1}{10}\min_{1 \leq i \leq m} |\mathcal I_{0,i}|.$ 

We can then construct a family of cut-off functions $\{ \chi_j \}_{j=1}^{m+k+\wt n}$ forming a smooth partition of unity of $\mathbb T$ such that each $\chi_j$ is supported in $\wt{\mathcal I}_j,$ with $\wt{\mathcal I}_j$ being one of the intervals from $\{ \mathcal I_i\}_{i=1}^k \cup\{\mathcal I_{0,i}\}_{i=1}^m \cup \{ \mathcal I_{a,i}\}_{i=1}^{\wt n}.$ The choice of the intervals ensures that $\chi_j \equiv 1$ near $(Z_+)^{-1}(c).$ 

See Figure \ref{Figure 1} for the partition of $[-\pi,\pi]$. 
\begin{figure}
\begin{tikzpicture}[thick, scale=0.45]
\draw[very thin,color=gray];
\draw[->] (-4,0) -- (24,0) node[right] {$y$};
\draw[->] (0,-4) -- (0,16) node[above] {$c$};
\draw[red,thick] (-1,10) -- (23,10);
\draw (23.5,10) node {\small $c$};
\draw[] (-1,-0.5) node {$-\pi$};
\draw (0,7) parabola bend  (1.2,5) (3,10);
\draw[dotted](3,10)-- (3,0); 
\draw[dotted](2.5,7.5)-- (2.5,0);
\draw[dotted](3.5,12)-- (3.5,0);
\draw (3,-0.5) node {\small $y_{c,1}$}; 
\draw (3,-1.5) node {\small $\cI_1$};
\draw[->] (2.5,0) -- (3.5,0);
\draw[->]  (3.5,0)--(2.5,0) ;
\draw (1.25,-3) node {\small $\cI_{a,1}$};
\draw (3,10) parabola bend  (4.1,13) (5,11);
\draw (5,11) parabola bend  (5.5,10) (6.5,12) ;
\draw[dotted](5.5,10)-- (5.5,0); 
\draw[dotted](5,11)-- (5,0);
\draw[dotted](6,10.5)-- (6,0);
\draw[->](5,0) -- (6,0);
\draw[->] (6,0)--(5,0);
\draw (5.5,-0.5) node {\small $y_{0,1}$};
\draw (5.5,-1.5) node {\small $\cI_{0,1}$};
\draw (4.25,-3) node {\small $\cI_{a,2}$};
\draw (6.5,12) parabola bend  (7,12.6) (8,10) ;
\draw[dotted](8,10)-- (8,0); 
\draw[dotted](8.5,8.1)-- (8.5,0);
\draw[dotted](7.5,12)-- (7.5,0);
\draw[->](7.5,0)-- (8.5,0); 
\draw[->] (8.5,0)--(7.5,0);
\draw (8,-0.5) node {\small $y_{c,2}$}; 
\draw (8,-1.5) node {\small $\cI_2$};
\draw (6.75,-3) node {\small $\cI_{a,3}$};
\draw (8,10)  parabola bend  (9.5,7) (10.5,8.5) ;
\draw (10.5,8.5)  parabola bend  (11.5,10) (12.5,8.5) ;
\draw[dotted](11.5,10)-- (11.5,0); 
\draw[dotted](12,9.6)-- (12,0);
\draw[dotted](11,9.5)-- (11,0);
\draw[->](11,0)-- (12,0); 
\draw[->] (12,0)--(11,0);
\draw (11.5,-0.5) node {\small $y_{0,2}$};
\draw (11.5,-1.5) node {\small $\cI_{0,2}$};
\draw (9.75,-3) node {\small $\cI_{a,4}$};
\draw (12.5,8.5)  parabola bend  (13.5,7.5) (15,10) ;
\draw[dotted](15,10)-- (15,0); 
\draw[dotted](14.5,8.5)-- (14.5,0);
\draw[dotted](15.5,11.5)-- (15.5,0);
\draw[->](15.5,0)-- (14.5,0); 
\draw[->] (14.5,0)--(15.5,0);
\draw (15,-0.5) node {\small $y_{c,3}$};
\draw (15,-1.5) node {\small $\cI_3$};
\draw (13.5,-3) node {\small $\cI_{a,5}$};
\draw (15,10) parabola bend  (17.5,14) (21,7) node[above,right]{\small Graph of $u+b$} ;
\draw[dotted](20.1,10)-- (20.1,0); 
\draw[dotted](19.7,11)-- (19.7,0);
\draw[dotted](20.4,8.7)-- (20.4,0);
\draw[->](19.7,0)-- (20.4,0); 
\draw[->] (20.4,0)--(19.7,0);
\draw (20,-0.5) node {\small $y_{c,4}$};
\draw (20,-1.5) node {\small $\cI_4$};
\draw (17.5,-3) node {\small $\cI_{a,6}$};
\draw (21.2,-0.4) node {\small $\pi$};
\draw[dotted](21,7)-- (21,0);
\draw (21,-3) node {\small $\cI_{a,7}$};
\end{tikzpicture}
 \caption{Partition of $[-\pi,\pi]$}
 \label{Figure 1}
\end{figure}
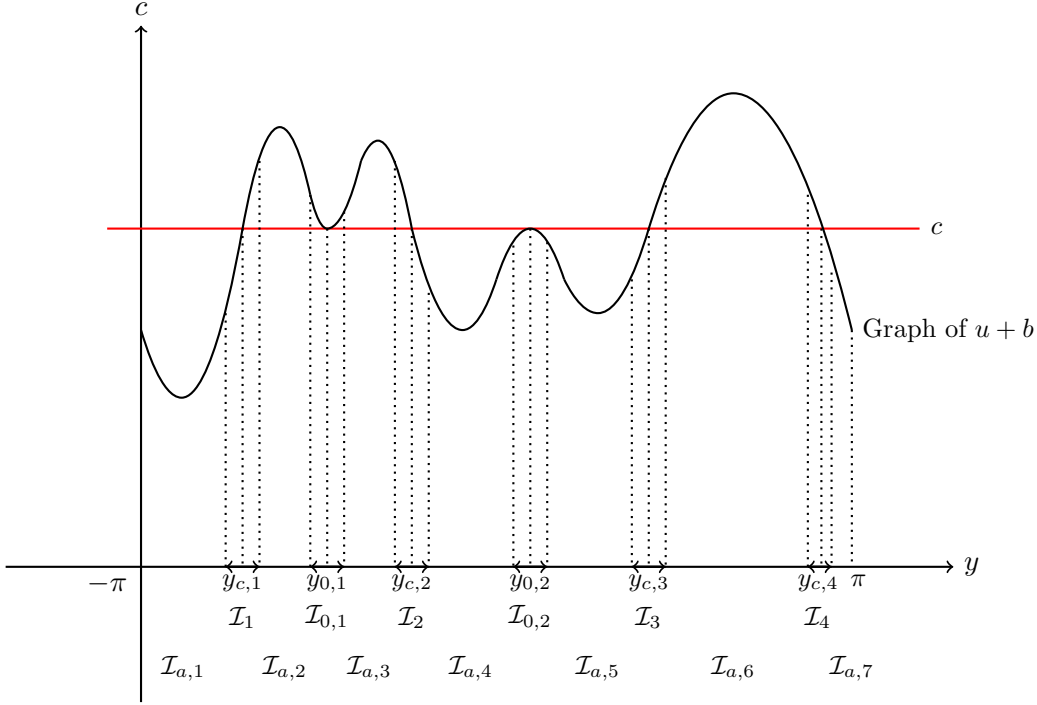

The proofs of the propositions are as follows.

\noindent \textit{Proof of Proposition \ref{prop: lap} } We note that on the intervals $\{\mathcal I_{a,i} \}_{i=1}^{\wt n}$ we do not encounter any issue in integration by parts. As a consequence of Lemma \ref{lem: m}, $q(y_c) = q' (y_c)= 0$ implies that $\pa_y \Phi(y_c)=0,$ allowing us to set the test function to be $f = \text{sgn}(Z_+ -c) \Phi \chi_j$ on $\wt{\mathcal I}_j$ if $\wt{\mathcal  I}_j \in \{ \mathcal I_i  \}_{i=1}^k,$ while Lemma \ref{lem: i} and Lemma \ref{lem: r} have enabled us also to set $\psi= \text{sgn}(Z_+ -c) \Phi \chi_j$ on $\wt{\mathcal I}_j$ when $\wt{\mathcal I}_j \in \{\mathcal I_{0,i} \}_{i=1}^m.$

Summing all of the integral identities and noticing that the terms containing $\chi'_{j}$ cancel each other,  we have
\begin{equation}\notag
\int_{-\pi}^\pi |Z_+ -c|(Z_- -c) \sum_{j=1}^{m+k} \chi_j \left(|\partial_y \Phi|^2 +\alpha^2 |\Phi|^2 \right)\mathrm{d}y = 0,
\end{equation}
that is,
\begin{equation}\notag
\int_{-\pi}^\pi |Z_+ -c|(Z_- -c)  \left(|\partial_y \Phi|^2 +\alpha^2 |\Phi|^2 \right)\,\mathrm{d}y = 0.
\end{equation}
Hence, the fact that $(Z_- -c)$ is sign-definite implies that $\Phi \equiv 0$, and $\Phi_n\rightharpoonup 0$ in $L^2$. 

It remains to be shown that $\{ \Phi_n \}_{n=1}^\infty$ and $\{ (Z_+ -c_n)(Z_- -c_n)\partial_y \Phi_n \}_{n=1}^\infty$ converge strongly in $L^2$ and $H^1,$ respectively, which is clear outside the set $(Z_+)^{-1}(c).$ Moreover, by Lemma \ref{lem: m} we already know that the desired strong convergence result holds outside the critical points $y_{0,1},  y_{0,2}, ..., y_{0, k}.$

As for the critical points in $(Z_+)^{-1}(c),$ i.e., $y_{0,1},  y_{0,2}, ..., y_{0, k},$ if $\{\Phi_n,c_n,F_n\}_{n=1}^\infty$ satisfies the conditions in Lemma \ref{lem: i}, then by \eqref{eq:bdd1} and \eqref{eq: strongto0}, $$\Phi_n\to \Phi\equiv 0 \text{ in }L^2(\mathcal{I}_0).$$
Recalling $q_n = (Z_+ -c_n)(Z_- -c_n)\partial_y \Phi_n,$ which satisfies Equation \eqref{eq:eqn2}, i.e., 
$$q''_n  -\alpha^2 q_n = F'_n + \alpha^2 \left( Z'_+(Z_- -c_n )+ Z'_-(Z_+ - c_n ) \right)\Phi_n,$$ 
we have, from our assumptions $\|\Phi_n\|_{L^2}+ \|q_n\|_{H1}=1$ and $F_n \to 0$ in $H^1,$ that 
\begin{equation}\notag
\|q''_n\|_{L^2} \lesssim \|\Phi_n\|_{L^2} + \|q_n\|_{L^2}  + \|F_n \|_{H^1},
\end{equation}
which together with the week convergence $q_n\rightharpoonup q$ implies $q_n \rightarrow q$ in $H^1(\mathcal{I}_0),$ that is, 
$$(Z_+ -c_n)(Z_- -c_n)\partial_y \Phi_n \rightarrow (Z_+ -c)(Z_- -c) \partial_y \Phi \equiv 0\text{ in }H^1(\mathcal I_0).$$ 

Finally, if $\{\Phi_n,c_n,F_n\}_{n=1}^\infty$ satisfies the conditions in Lemma \ref{lem: r} instead, we recall Equation \eqref{eq:2}, in which both $y_1, y_2\notin (Z_+)^{-1}(c)\cup (Z_-)^{-1}(c)$ are chosen uniform in $n$.
Then by \eqref{eq: strongto0}, we obtain that as $n \to \infty$,
$
|\Phi_n(y_1)|+|\Phi_n(y_2)| \rightarrow 0,
$
and thus $\|F^*_n\|_{L^{\infty}} \rightarrow 0$. 
Lemma \ref{lem: phi*} then ensures that $\Phi_n \rightarrow 0$ in $L^2(\mathcal I_0).$ By the same argument as above, we obtain that $q_n \rightarrow q$ in $H^1(\mathcal{I}_0)$.

Thus, we have shown that $\Phi_n \rightarrow 0$ in $L^2$ and $(Z_+ -c_n)(Z_- -c_n)\pa_y \Phi_n \rightarrow 0$ in $H^1,$ which contradicts the assumption that $\|\Phi_n\|_{L^2} + \|(Z_+ -c_n)(Z_- -c_n) \pa_y\Phi_n\|_{H^1} =1.$ As the same argument applies to the case $c \in \mathrm{Ran}\,(Z_-),$ it must be that the uniform estimate 
$$ \|\Phi (\cdot, c)\|_{L^2} + \|(Z_+ -c)(Z_- -c) \partial_y \Phi(\cdot, c) \|_{H^1} \leq C \|F(\cdot, c)\|_{ H^1}$$
holds true for $c \in \left(\Omega_{\ep_0}\setminus \left( \text{Ran}\, Z_+ \cup \text{Ran}\, Z_+\right)\right)$.  

\qed

\noindent \textit{Proof of Proposition \ref{prop: conti} } Suppose that there exists some $c \in \text{Ran}\, Z_+$ for which the limit $\Phi^+(y, c)$ do not exist, then we can find two sequences $\{c_{n,1} \}_{n=1}^\infty$ and $\{ c_{n,2} \}_{n=1}^\infty$ such that $\mathrm{Im}\,c_{n,j} >0$ and $\mathrm{Re}\,c_{n,j} =c,$ $j=,1,2,$ $|c_{n,1}-c_{n,2}| \to 0$ as $n \to \infty,$ while there exists some $\d >0$ such that 
$$\left\| \Phi(\cdot,  c_{n,1}) -\Phi(\cdot, c_{n,2}) \right\|_{L^{r}} \geq \d,  \ 1< r<2, \ \forall n \in \mathbb N.$$ 

Up to a subsequence, $c_{n,j} \to c \in \text{Ran}\, Z_+.$  By virtue of Proposition \ref{prop: lap}, $$\| \Phi(\cdot, c_{n,j}) \|_{L^2} + \|(Z_+ -c_{n,j})(Z_- -c_{n,j})\pa_y \Phi(\cdot, c_{n,j}) \|_{H^1} \leq \|F(\cdot, c_{n,j})\|_{H^1} <C ,$$ from which we know by \eqref{eq:eqn2} that 
\beno
(Z_+ -c_{n,j})(Z_- -c_{n,j})\pa_y \Phi(\cdot, c_{n,j})\in H^2,
\eeno
and that there exist $\Phi_{c,j} \in L^2$ with $\left( (Z_+ -c)(Z_- -c)\pa_y \Phi_{c,j} \right) \in H^2,$ $j=1,2,$ such that up to a subsequence, 
\beno
\Phi(\cdot, c_{n,j}) \rightharpoonup \Phi_{c, j} (\cdot) \quad \text{in}\quad L^2,\quad j=1,2.
\eeno

As before, the situation at points in the monotone regions and that at the critical points have to be discussed separately.   

As for $y_{c} \in (Z_+)^{-1}(c)$ at which $Z_+$ is strictly monotone, from the proof of Lemma \ref{lem: m} we have, for $\forall  f \in C^{1}_0 (\mathcal I)$ and $j=1,2,$ that
\begin{equation}\notag
\begin{split}
&-\text{p.v.}\int_{\mathcal I} \frac{ \left( q'_{c, j}  f' + \alpha^2 q_{c, j} f \right)(y) }{(Z_+(y)-c)(Z_-(y)-c)} \mathrm{d}y + i\pi\frac{ \left( q'_{c, j}   f' + \alpha q_{c, j}   f \right) (y_c)}{ 2b(y_c) Z_+'(y_c)}\\
&=  -\text{p.v.}\int_{\mathcal I} \frac{F(y,c) f'(y)}{(Z_+(y)-c)(Z_-(y)-c)} \mathrm{d}y + i\pi \frac{F(y_c, c)  f'(y_c) }{2b(y_c)Z_+'(y_c)},
\end{split}
\end{equation}
where $q_{c, j}: = (Z_- -c)(Z_+ -c) \pa_y \Phi_{c, j},$ $j=1,2.$ Subtracting, we have
$$-\text{p.v.}\int \frac{ \left( (q'_{c,1}- q'_{c, 2})  f' + \alpha^2 (q_{c,1}-q_{c,2})  f \right)(y) }{(Z_+(y)-c)(Z_-(y)-c)} \mathrm{d}y + i\pi\frac{ \left( (q'_{c,1} -q'_{c,2})  f' + \alpha^2 (q_{c,1} -q_{c,2}) f \right) (y_c)}{ 2b (y_c) Z_+'(y_c)} = 0,$$
whose imaginary part ensures that 
$$\left(q_1 -q_2\right) (y_c) = \left( q'_1 -q'_2 \right)(y_c)=0.$$ 
By Hardy's inequality, we know that $\partial_y (\Phi_{c,1} - \Phi_{c,2}) \in L^2(\mathcal I).$

In the case of a critical point $y_0 \in \mathcal I_0,$  from Lemma \ref{lem: i} and Lemma \ref{lem: r} we know that $(y - y_0)\partial_y (\Phi_{c,1} - \Phi_{c,2}) \in L^2(\mathcal I_0)$ and
\begin{equation}\notag
-\int_{\mathcal I_0} (Z_+ -c)(Z_- -c) \left( \partial_y \Phi_{c,j}  f' + \alpha^2 \Phi_{c,j}  f \right) \mathrm{d}y = \int_{\mathcal I_0} F f \mathrm{d}y, \ \forall  f \in H^1_{0,w}(\mathcal I_0), \ j =1,2,
\end{equation}
which implies that
\begin{equation}\notag
-\int_{\mathcal I_0} (Z_+ -c)(Z_- -c) \left( (\partial_y \Phi_{c,1}- \partial_y \Phi_{c,2})  f' + \alpha^2 (\Phi_{c,1}-\Phi_{c,2})  f \right) \mathrm{d}y = 0, \ \forall  f \in H^1_{0,w}(\mathcal I_0).
\end{equation}
By setting $ f = \text{sgn}\left( (Z_+ -c)(Z_- -c)\right) \left( \Phi_{c,1}-\Phi_{c,2} \right)$ and integrating by parts, similar to the procedure in the proof of Proposition \ref{prop: lap}, we can then conclude that $\Phi_{c,1}-\Phi_{c,2}\equiv 0$ and 
\beno
\left(\Phi(\cdot, c_{n,1}) - \Phi(\cdot, c_{n,2}) \right) \rightharpoonup 0 \quad \text{in}\quad L^2.
\eeno

We then aim to show that the weak convergence mentioned above are in fact strong. Away from the critical points in $(Z_+)^{-1}(c),$ this is ensured by Lemma \ref{lem: m}. 

Near $y_0,$ which can be any of the critical points $y_{0,1}, y_{0,2},..., y_{0,k},$ the situation resembles those covered in Lemma \ref{lem: i}, i.e., $|\text{Im}\,c_{n,j}| > |\text{Re}\, c_{n,j} -c |$ for $j=1,2$. (Note here $\text{Re}\, c_{n,j}=c$.) Therefore, by Lemma \ref{lem: i}, we have 
\beno
(y-y_0)\pa_y \Phi(y,c_{n,j})\in L^2,\quad j=1\ \text{or}\ 2,
\eeno
which by the compactness result in Appendix \ref{rmk1}, yields the strong convergence locally near $y_0$ in this case, i.e.
$$\Phi(\cdot, c_{n,j}) \rightarrow \Phi_{c,j}(\cdot) \text{ in } L^r, \ 1< r <2, \ j=1, 2.$$

Thus in both cases we have $\Phi(\cdot, c_{n,1})-\Phi(\cdot, c_{n,2})\to \Phi_{c,1}(\cdot)-\Phi_{c,2}(\cdot)=0$ in $L^r$, which leads to a contradiction. 
\qed

\section{Linear damping and depletion}
Recalling that $\Psi_1 = (u-c) \Phi + \frac{\widehat  \phi_0}{b}$ and $\Phi_1 = b \Phi$ we have, by the formula in \eqref{eq: psi and phi-1} and Proposition \ref{prop: lap}, that 
\begin{equation}\notag
\begin{split}
\begin{pmatrix}\widehat  \psi \\
\widehat  \phi 
\end{pmatrix}(t, \alpha, y)
= & \lim_{\ep \to 0^+} \frac{1}{2\pi i}\int_{\partial \Omega_\ep} e^{-i\alpha t c} \begin{pmatrix}  \Psi_1 \\
\Phi_1 \end{pmatrix} (\alpha, y, c) \,\mathrm{d} c \\
= &  \lim_{\ep \to 0^+}\frac{1}{2\pi i} \bigg(\int_{\text{Ran}\,(u+b)  \cup \text{Ran}\,(u-b)} e^{-i\alpha t {(c-i\ep)}} \begin{pmatrix}  u-(c-i\ep) \\ 
b 
\end{pmatrix} \Phi(\al,y,c-i\ep) \,\mathrm{d}c\\
&\qquad\qquad \quad -\int_{\text{Ran}\,(u+b)  \cup \text{Ran}\,(u-b)} e^{-i\alpha t {(c+i\ep)}} \begin{pmatrix}  u- (c+i\ep)  \\
b \end{pmatrix} \Phi(\al,y,c+i\ep) \,\mathrm{d}c\bigg).
\end{split}
\end{equation}

For $c \in \left( \text{Ran}\,(u+b) \cup \text{Ran}\,(u-b) \right),$ we denote $c_{\ep}:=c+i\ep$ with $-\ep_0<\ep<\ep_0$. We recall that $\Phi(\al,y,c_{\ep})$ solves 
\begin{equation}
\partial_y\left(\left( (u-c_{\ep})^2-b^2\right) \partial_y \Phi(\al,y,c_{\ep}) \right) -\alpha^2 \left( (u-c_{\ep})^2-b^2\right)\Phi(\al,y,c_{\ep})  = F(\al,y,c_{\ep}),
\end{equation}
which can also be written as
\begin{equation}\notag
\partial_y\left((Z_+ -c_{\ep} ) (Z_- -c_{\ep})\partial_y \Phi(\al,y,c_{\ep}) \right) -\alpha^2 (Z_+ -c_{\ep} ) (Z_- -c_{\ep})\Phi(\al,y,c_{\ep}) = F(\al, y, c_{\ep}).
\end{equation}
Let $\Phi^{\pm}(\al,y,c)=\lim\limits_{\ep\to 0^+}\Phi(\al,y,c\pm i\ep),$ as defined in Proposition \ref{prop: conti}. For convenience, let us denote $\wt \Phi(\al, y, c):= \Phi^-(\al, y, c) -\Phi^+(\al, y, c).$ Then by Proposition \ref{prop: lap} and Proposition \ref{prop: conti}, we have
\begin{equation}\label{rslv}
\begin{split}
\begin{pmatrix}\widehat  \psi \\
\widehat  \phi 
\end{pmatrix}(t, \alpha, y)
= & \lim_{\ep \to 0^+} \frac{1}{2\pi i}\int_{\partial \Omega_{\ep}} e^{-i\alpha t c} \begin{pmatrix}  \Psi_1 \\
\Phi_1 \end{pmatrix} (\alpha, y, c) \,\mathrm{d} c \\
= &  \frac{1}{2\pi i} \int_{\text{Ran}\,(u+b) \cup \text{Ran}\,(u-b)} e^{-i\alpha t c} \begin{pmatrix}  (u-c) \wt\Phi \\
b \wt \Phi \end{pmatrix} (\alpha, y, c) \,\mathrm{d}c. 
\end{split}
\end{equation}

\noindent \textit{Proof of Theorem \ref{main thm2} } Differentiating \eqref{rslv} in $t$ yields
\begin{equation}
\partial_t  \begin{pmatrix}\widehat  \psi \\
\widehat  \phi 
\end{pmatrix}(t, \alpha, y) = \frac{1}{2\pi i} \int_{\text{Ran}\,(u+b) \cup \text{Ran}\,(u-b)} i\alpha c e^{-i\alpha t c}  \begin{pmatrix}  (c-u) \wt\Phi \\
-b \wt \Phi \end{pmatrix} (\alpha, y, c) \,\mathrm{d}c.
\end{equation}

By Plancherel's theorem, we have the following estimates --
\begin{equation}\notag
\begin{split}
\left\|   \begin{pmatrix}\widehat  \psi \\
\widehat  \phi 
\end{pmatrix} \right\|_{L^2_t L^2_y}^2 
+  \left\| \partial_t \begin{pmatrix}\widehat  \psi \\
\widehat  \phi 
\end{pmatrix} \right \|_{L^2_t L^2_y}^2 = & \int_{\mathbb T} \int_{-\infty}^\infty\left( \left|  \begin{pmatrix}\widehat  \psi \\
\widehat  \phi 
\end{pmatrix} \right|^2 + \left|\partial_t  \begin{pmatrix}\widehat  \psi \\
\widehat  \phi 
\end{pmatrix} \right|^2 \right)(t, \alpha, y)  \,\mathrm{d}t\,\mathrm{d}y\\
=& \int_{\mathbb T} \int_{\text{Ran}\,(u+b) \cup \text{Ran}\,(u-b)} (1+(\alpha c)^2)\left| \begin{pmatrix}  (u-c) \wt \Phi \\
b \wt \Phi \end{pmatrix}(\alpha, y, c) \right |^2 \mathrm{d}c\,\mathrm{d}y.
\end{split}
\end{equation}
Invoking Proposition \ref{prop: conti} and the boundedness of $b,$ we have
\begin{equation}
\begin{split}
\left\|   \begin{pmatrix}\widehat  \psi \\
\widehat  \phi 
\end{pmatrix} \right\|_{L^2_t L^2_y}^2 
+  \left\| \partial_t \begin{pmatrix}\widehat  \psi \\
\widehat  \phi 
\end{pmatrix} \right \|_{L^2_t L^2_y}^2 \leq & C_{\al} \int_{\text{Ran}\,(u+b) \cup \text{Ran}\,(u-b)} \|\wt \Phi(\alpha, \cdot, c) \|_{L_y^2}^2 \, \mathrm{d}c \\
\leq & C_{\alpha} \|F\|_{H^1_y}^2 \lesssim \left \| \begin{pmatrix}  \wh{\psi}_0 \\
 \wh{\phi}_0 \end{pmatrix} \right\|_{H^3_y}^2.
\end{split}
\end{equation}

\qed

\noindent \textit{Proof of Theorem \ref{main thm3} } Let $y_0$ be any critical point of $u+b$ or $u-b$. By the conclusion of Lemma \ref{lem: i}, for $c \in \text{Ran}\,(u+b) \cup \text{Ran}\,(u-b)$, it holds that
\begin{gather*}
\left| \Phi (\alpha, y_0, c \pm i\ep) \right| \leq C \big|\left( (u+b)(y_0) - (c\pm i\ep) \right) \left( (u-b)(y_0) - (c\pm i\ep) \right)\big|^{-\frac{1}{4}},\\
\left| \partial_y \Phi (\alpha, y_0, c \pm i\ep) \right| \leq C \big|\left( (u+b)(y_0) - (c\pm i\ep) \right) \left( (u-b)(y_0) - (c\pm i\ep) \right)\big|^{-\frac{3}{4}},
\end{gather*}
which implies uniform bounds on both $\Phi(\al, y_0, \cdot \pm i\ep_0)$ in $L_c^\rho,$ $\rho \in [1,4)$ and $ \partial_y \Phi (\alpha, y_0, \cdot \pm i\ep) $ in $L^p_c,$ $p \in [1, \frac{4}{3}).$ Thus, there exists a subsequence $\ep_n \to 0^+$ as well as $\Lambda^\pm \in L_c^\rho$ and $ \Theta^{\pm}\in L^p_c$ such that as $\ep_n \to 0^+$,
\begin{gather*}
\Phi(\alpha, y_0, \cdot \pm i\ep_n) \rightharpoonup \Lambda^\pm (\alpha, y_0, \cdot),\\
\partial_y \Phi(\alpha, y_0, \cdot \pm i\ep_n) \rightharpoonup \Theta^\pm (\alpha, y_0, \cdot).
\end{gather*} 

By \eqref{rslv}, we have
\begin{align*}
\begin{pmatrix} \widehat{U}_1\\
\widehat{H}_1
\end{pmatrix}(t, \alpha, y_0) = & \partial_y \begin{pmatrix} \widehat{ \psi}\\
\widehat{ \phi}
\end{pmatrix}(t, \alpha, y_0)\\= & \lim_{\ep_n \to 0^+} \frac{1}{2 \pi i} \int_{\partial \Omega_{\ep_n}} e^{-i\alpha tc} \partial_y\begin{pmatrix}\Psi_1 \\
\Phi_1
\end{pmatrix}(\alpha, y_0, c)\,\mathrm{d}c\\
= & \lim_{\ep_n \to 0^+} \frac{1}{2 \pi i} \int_{\partial \Omega_{\ep_n}} e^{-i\alpha tc} \begin{pmatrix}(u-c)\pa_y \Phi+ u'\Phi \\
b\pa_y \Phi+b' \Phi
\end{pmatrix}(\alpha, y_0, c)\,\mathrm{d}c\\
=& \frac{1}{2\pi i} \int_{\text{Ran}\,(u+b) \cup \text{Ran}\,(u-b)} e^{-i\alpha t c} \begin{pmatrix} (u-c) (\Theta^--\Theta^+) \\
b(\Theta^--\Theta^+)  \end{pmatrix}(\alpha, y_0, c) \,\mathrm{d} c\\
& + \frac{1}{2\pi i} \int_{\text{Ran}\,(u+b) \cup \text{Ran}\,(u-b)} e^{-i\alpha t c} \begin{pmatrix} u' (\Lambda^- - \Lambda^+) \\
b'  (\Lambda^- - \Lambda^+)  \end{pmatrix}(\alpha, y_0, c) \,\mathrm{d} c.
\end{align*}
The desired conclusion follows from Riemann-Lebesgue lemma, as $(\Theta^--\Theta^+)(\alpha, y_0, \cdot)\in L_c^1$ and $(\Lambda^- - \Lambda^+)(\alpha, y_0, \cdot)\in L_c^1$. 

\qed

\appendix

\section{The homogeneous Sturmian equation}\label{sturm}
In this section, we shall first construct a regular solution to the homogeneous Sturmian equation on $[y_1,y_2]$ which contains only one critical point $y_0$, i.e. $y_1<y_0<y_2$:
\ben\label{homo}
\pa_y\big((Z_--c)(Z_+-c)\pa_y\va\big)-\al^2(Z_--c)(Z_+-c)\va=0,
\een
for $c$ in an $\ep_0$-strip $\mathcal{S}_{\ep_0}$ containing $Z_+(y_0)$: 
\beno
\mathcal{S}_{\ep_0}=\Big\{c_r+i\ep:\ c_r\in \big[\min\{Z_+(y_0),Z_+(y_1)\},\max\{Z_+(y_0),Z_+(y_1)\}\big],\ |\ep|<\ep_0\Big\}
\eeno 
with $Z_+'(y_0)=0$ and $Z_+(y_1)=Z_+(y_2)$. 

One may follow the same argument and construct a regular solution for the case $c$ in an $\ep_0$-strip $\mathcal{S}_{\ep_0}$ containing $Z_-(y_0)$ with $Z_-'(y_0)=0$. We omit the details here. 

 We note that $
Z_+(y)\geq C^{-1}>0>-C^{-1}\geq Z_-(y)$, $Z_+(y_1)=Z_+(y_2),$ $Z'_+(y_0)=0$ for the critical point $y_0\in(y_1,y_2)$ and $|Z'_+(y)|>0$ for $y\neq y_0$.  
Let $$D_0:=\{c\in[\min\{Z_+(y_0),Z_+(y_1)\},\max\{Z_+(y_0),Z_+(y_1)\}]\}$$ and 
\beno
&\ & D_{\ep_0}:=\left\{c=c_r+i\ep, \ \ c_r\in[\min\{Z_+(y_0),Z_+(y_1)\},\max\{Z_+(y_0),Z_+(y_1)\}], 0<|\ep|<\ep_0 \right\}.
\eeno
Then $\mathcal{S}_{\ep_0}=D_0\cup D_{\ep_0}$. 

Given $c_r \in [\min\{Z_+(y_0),Z_+(y_1)\},\max\{Z_+(y_0),Z_+(y_1)\}],$ when restricted to $[y_0,y_2],$  we can find $y^r \in [y_0, y_2]$ such that $Z_+(y^r)=c_r.$ And when restricted to $[y_1,y_0],$ we can find $y^{\ell}\in [y_1,y_0]$ such that $Z_+(y^{\ell})=c_r.$

\begin{proposition}\label{prop: sol. hom. y2}
1.  For $c\in \cS_{\ep_0}$, there exists a solution
$\va^r(y,c)\in C([y_0,y_2]\times\cS_{\ep_0})$ of the
Sturmian equation \eqref{eq: homo eq} and $\pa_y\va^r(y,c)\in C([y_0,y_2]\times \cS_{\ep_0})$.
 Moreover, there exists $\ep_1>0$ such that for any $\ep_0\in[0,\ep_1)$
  and $(y,c)\in [y_0,y_2]\times \cS_{\ep_0}$,
\beqno
|\va^r(y,c)|\geq \f{1}{2}, \ \ |\va^r(y,c)-1|\leq C|y-y^r|^2,
\eeqno
where the constants $\ep_1, C$ may depend on $\al$.\\
2.  For $c\in D_0$, for any $y\in[y_0,y_2]$, there is a constant $C$(depends on $\al$) such that,
\beno
\va^r(y,c)\geq \va^r(y',c)\geq 1, \ \ \text{for} \quad y_0\leq y^r\leq y'\leq y\leq y_2
\ \ \text{or}\ \ y_0\leq y\leq y'\leq y^r\leq y_2.
\eeno
\end{proposition}
\begin{proposition}\label{prop: sol. hom. y1}
1.  For $c\in \cS_{\ep_0}$, there exists a solution
$\va^\ell(y,c)\in C([y_1,y_0]\times\cS_{\ep_0})$ of the
Sturmian equation \eqref{eq: homo eq} and $\pa_y\va^\ell(y,c)\in C([y_1,y_0]\times \cS_{\ep_0})$.
 Moreover, there exists $\ep_1>0$ such that for any $\ep_0\in[0,\ep_1)$
  and $(y,c)\in [y_1,y_0]\times \cS_{\ep_0}$,
\beqno
|\va^{\ell}(y,c)|\geq \f{1}{2}, \ \ |\va^{\ell}(y,c)-1|\leq C|y-y^{\ell}|^2,
\eeqno
where the constants $\ep_1, C$ may depend on $\al$.\\
2.  For $c\in D_0$, for any $y\in[y_1,y_0]$, there is a constant $C$(depends on $\al$) such that,
\beno
\va^{\ell}(y,c)\geq \va^{\ell}(y',c)\geq 1, \ \ \text{for} \quad y_1\leq y^{\ell}\leq y'\leq y\leq y_0
\ \ \text{or}\ \ y_1\leq y\leq y'\leq y^{\ell}\leq y_0.
\eeno
\end{proposition}
In the following, we only give the proof of Proposition \ref{prop: sol. hom. y2} and Proposition \ref{prop: sol. hom. y1} can be similarly proved. To prove the existence result for Equation \eqref{homo},  we introduce the following adapted norms. 
\begin{definition} For a function $f(y, c)$ on $[y_0,  y_2] \times \cS_{\epsilon_0}$, we define
\begin{gather*}
\|f\|_{X_0} := \sup_{(y, c) \in [y_0,  y_2] \times \cS_{\ep_0}}\left| \frac{ f(y,c)}{\cosh(A(y-y^r))} \right|,\quad
\|f\|_{Y_0}:= \|f\|_{X_0} + \frac{1}{A} \|\partial_y f \|_{X_0}.
\end{gather*}
\end{definition}

In order to give the solution formula,  we introduce the following integral operators.
\begin{definition}
Let $y\in[y_0,y_2].$ The Sturmian integral operator $S$ is defined by
\beqno
Sf(y,c):=S_0\circ S_1f(y,c)=\int_{y^r}^y
\f{\int_{y^r}^{y'}(Z_-(y'')-c)(Z_+(y'')-c)f(y'',c)\mathrm{d}y''}
{(Z_-(y')-c)(Z_+(y')-c)}\mathrm{d}y',
\eeqno
where
\begin{gather*}
S_0f(y,c):=\int_{y^r}^yf(y',r)\mathrm{d}y',  \ S_1f(y,c):=
\f{\int_{y^r}^{y}(Z_-(y'')-c)(Z_+(y'')-c)f(y'',c)\mathrm{d}y''}
{(Z_-(y)-c)(Z_+(y)-c)}.
\end{gather*}
\end{definition}

\begin{lemma}\label{lem:l1}
There exists a constant $C$ independent of $A$ such that
$$\|S_0 f\|_{X_0} \leq \frac{C}{A}\|f\|_{X_0},  \ \|S_{1} f\|_{X_0} \leq \frac{C}{A}\|f\|_{X_0},  \ \|S f \|_{X_0} \leq \frac{C}{A^2}\|f\|_{X_0}. $$
Furthermore,  there holds
\begin{gather*}
\| Sf\|_{Y_0} \leq \frac{C}{A^2} \|f\|_{Y_0}.
\end{gather*}
\end{lemma}

\begin{proof} For $c\in D_0$, we shall only prove the part of the lemma for $Z_+(y_0)=c,$ as the proof when $Z_-(y_0) =c$ is along the same lines. As $\text{Ran}\,Z_+ \cap \text{Ran}\, Z_- =\emptyset,$ we can find some positive $C$ such that $C^{-1} < |Z_-(y)-c|  <C$ for $y\in[y_0,y_2]$. Firstly, by definition,  we have
\begin{equation}\label{eq:ne1}
\begin{split}
\|S_0 f \|_{X_0} = & \sup_{(y, c) \in [y_0, y_2] \times D_0} \left | \frac{1}{\cosh(A(y-y^r))}\int^{y}_{y^r} {\cosh(A(y'-y^r))} \frac{f(y', c)}{{\cosh(A(y-y^r))}} \mathrm{d} y' \right| \\
\leq & \sup_{(y, c) \in [y_0,  y_2] \times D_0} \left|\frac{1}{\cosh(A(y-y^r))}\int^{y}_{y^r} {\cosh(A(y'-y^r))} \mathrm{d} y'  \right| \|f\|_{X_0}\\
\leq & \frac{C}{A} \|f\|_{X_0}.
\end{split}
\end{equation}

And due to $|Z_+(y'')-c| \leq |Z_+(y)-c|$ for $y_0 \leq y'' \leq y \leq y_2,$ 
\begin{equation}\label{eq:ne2}
\begin{split}
\|S_{1} f \|_{X_0} \leq & \sup_{(y, c) \in [y_0, y_2] \times D_0} \left | \frac{1}{\cosh(A(y-y^r))} \int^{y}_{y^r} \cosh(A(z-y^r)) \frac{ f(y'',c)} {\cosh(A(z-y^r))}\mathrm{d}y''   \right |\\
\leq & \frac{C}{A}\|f\|_{X_0}.
\end{split}
\end{equation}

Composing inequalities \eqref{eq:ne1} and \eqref{eq:ne2},  we have 
\beq\label{es: S}
\|S f \|_{X_0} \leq \frac{C}{A^2}\|f\|_{X_0}.
\eeq

On the other hand, direct calculation shows  $\partial_y S f =  S_1 f.$
By \eqref{eq:ne2}, we have
\begin{equation}\notag
\begin{split}
\| \partial_y S f\|_{X_0} \leq \frac{C}{A}\| f\|_{X_0}, 
\end{split}
\end{equation}
and then, combining with \eqref{es: S}, it holds
\beno
\|Sf\|_{Y_0}\leq \f{C}{A^2}\|f\|_{Y_0}.
\eeno

In similar ways  we can prove the inequalities for $ c \in D_{\epsilon_0}$.
\end{proof}

\textbf{Proof of Proposition \ref{prop: sol. hom. y2}.}
For $c \in \cS_{\ep_0},$ we can solve the homogeneous Sturmian equation on $[y_0,y_2]$:
\beq\label{eq: homo eq}
\Big\{\begin{array}{l}
 \pa_y\Big((Z_--c)(Z_+-c)\pa_y\va^r\Big)=\al^2(Z_--c)(Z_+-c)\va^r,\\
\va^r(y^r,c)=1, \ \pa_y\va^r(y^r,c)=0,
\end{array}\Big.
\eeq
where $Z_+(y^r) = c_r,$ as previously defined.  Integrating twice yields $\va^r=1+\al^2S\va^r.$ We choose $A$ so that $\frac{C\alpha^2}{A^2} \leq \frac{1}{2} <1.$ Since $\|Sf\|_{Y_0} \leq \frac{C}{A^2} \|f\|_{Y_0},$  the operator $\left(I- \alpha^2S \right)$ is invertible in the adapted space $Y_0$ and the solution to equation \eqref{eq: homo eq} is given by $\va^r= \left(I- \alpha^2S \right)^{-1} 1.$ As $\|\va^r\|_{Y_0} \leq \| 1\|_{Y_0} + \alpha^2\| S\va^r \|_{Y_0} \leq C+ \frac{1}{2}\| \va^r\|_{Y_0},$ it holds that $\|\va^r\|_{Y_0} < C.$

We can rewrite $S$ as
\begin{equation}\notag
S f (y, c) = |y-y^r|^2 \int^1_0 \int^1_0 f ( y^r+ (y-y^r)st) K_0(s,  t,  y, c) \mathrm{d}s \mathrm{d}t,
\end{equation}
where $$K_0 (s, t, y, c) = t \frac{(Z_-(y^r + (y-y^r)st)-c)(Z_+(y^r + (y-y^r)st)-c)}{(Z_-(y^r+ (y-y^r)t)-c)(Z_+(y^r+ (y-y^r)t)-c)}.$$ Since $|K_0| \leq t$ and $K_0 \in C\left([y_0, y_2] \times \cS_{\ep_0}\right),$ $S$ maps $C([y_0, y_2] \times \cS_{\ep_0})$ to $C\left([y_0,y_2] \times \cS_{\ep_0}\right).$ Thus,  we can deduce that $\va^r(y,c) \in C([y_0,y_2] \times \cS_{\ep_0})$ from the formula $\va^r(y,c) =\sum\limits_{k=0}^\infty \alpha^{2k}S^k 1$ for $c \in \cS_{\ep_0}$ and the uniform convergence of the series.

Since $S$ is a positive operator,  $\va^r(y ,c) \geq 1$ for $c \in \cS_{\ep_0}.$  By the continuity of $\va^r(y, c)$,  there exists some $\ep \in (0,  \ep_0]$ such that for $c$ belonging to  $\cS_{\ep_0}$, it holds that
$$|\va^r(y, c)| > \frac{1}{2}.$$  

From the integral formula of $S,$ we have
\begin{equation}\notag
\begin{split}
\left| \va^r(y,c) -1 \right| \leq & \alpha^2 \int^{y}_{y^r} \int^{y'}_{y^r} \left| \va^r(z, c)\right| \left|\frac{(Z_-(y'')-c)(Z_+(y'')-c)}{(Z_-(y')-c)(Z_+(y')-c)}\right| \mathrm{d}y''\mathrm{d}y' \\
\leq & C\left(\alpha, \|\va^r \|_{L^\infty([y_0, y_2] \times \cS_{\ep_0})} \right) |y-y^r|^2.
\end{split}
\end{equation}

\section{A compactness lemma} \label{rmk1}
{We prove a useful compactness result.
\begin{lemma}
Let $\{y_n\}_{n=1}^\infty \subset \mathcal I:=[a, b]$ be such that $y_n \to \frac{a+b}{2}.$ Let $\{f_n\}_{n=1}^{\infty}$ be a family of functions defined on $\mathcal I$ satisfying the uniform bound
$$\|f_n\|_{L^{r}} + \|(y-y_n)\pa_yf_n\|_{L^p} \leq C, \ 1 < r < p < \infty,$$ 
then there exist $f_{\infty} \in L^p$ and a subsequence $\{f_{n_j}\}_{j=1}^{\infty}$ such that as $j \to \infty,$
\beno
f_{n_j}\rightarrow f_{\infty} \text{ in } L^r, \ 1< r < p.
\eeno
\end{lemma}
\begin{proof}
By virtue of Kolmogorov-Riesz theorem, the proof of the above compactness result amounts to showing that the family $\{f_n\}_{n=1}^\infty$ is equicontinuous in $L^r.$ Let $\mathcal I_h:=[y_n-h, y_n+h]$ and $\mathcal I_{2h}:=[y_n-2h, y_n+2h]$ (Note that $\mathcal I_{2h} \subset \mathcal I$ for sufficiently small $h$.) We have 
\begin{equation}\notag
\begin{split}
\|f_n(x+h) -f_n(x)\|_{L^r}^r =& \int_{\mathcal I \setminus \mathcal I_h} \left| \int_{x}^{x+h} f'_n(x')\, \mathrm{d} x' \right|^r \mathrm{d}x + \int_{\mathcal I_h} \left|f_n(x+h) - f_n(x)\right|^r\,\mathrm{d}x \\
:= & I_{*} + I_{**}.
\end{split}
\end{equation}

Noticing that $|x - y_n| \geq h$ for $x \in \mathcal I \setminus \mathcal I_h,$ we use Hardy-Littlewood maximal inequality to estimate $I_*$ as follows --
\begin{equation}\notag
\begin{split}
I_{*} = & \int_{\mathcal I \setminus \mathcal I_h} h^r \left| \frac{1}{h} \int_{x}^{x+h} (x' - y_n)f'_n(x')\cdot \left( \frac{1}{x' -y_n} \right) \,\mathrm{d} x' \right|^r \mathrm{d}x \\
\leq & C_p \int_{\mathcal I \setminus \mathcal I_h} \left( h \|(x - y_n)f_n \|_{L^p} \right)^r h^{-\frac{r}{p}}\, \mathrm{d} x \leq  C h^{r \left(1-\frac{1}{p} \right)}.
\end{split}
\end{equation}

As for $I_{**},$ we let $\chi$ be a smooth cut-off function such that $\chi \equiv 1$ on $\mathcal I_{2h}.$ Integration by parts yields the following identity --  
\begin{equation}\notag
\int_{\mathcal I} |f_n|^p \chi\, \mathrm{d}y= - \int_{\mathcal I} (y-y_n) |f_n|^p \chi' \,\mathrm{d}y  - \frac{p}{2} \int_{\mathcal I} \chi(y-y_n)\left(  f'_n |f_n|^{p} f_n^{-1} +   \pa_y \overline{f_n} |f_n|^{p} \overline {f_n}^{-1} \right) \mathrm{d}y,
\end{equation}
which, along with Sobolev inequality $\|f_n\|_{L^\infty(\mathcal I \setminus \mathcal I_{h})} \leq C \|f_n\|_{W^{1,q}(\mathcal I \setminus \mathcal I_{h})}, \ q >1,$ leads to 
\begin{equation}\notag
\|f_n\|_{L^p(\mathcal I_{h})}^p \leq C_p \|(y-y_n) \partial_y f_n \|_{L^p}^p + C_h \|f_n\|_{L^p(\mathcal I\setminus \mathcal I_{h})}^p \leq C.
\end{equation}
Hence, we can simply estimate $I_{**}$ by H\"older's inequality -- 
\begin{equation}\notag
I_{**} \leq  \int_{\mathcal I_h} \left( | f_n(x+h)|^r + |f_n(x)|^r \right) \mathrm{d}x \leq 2 \|f_n\|^r_{L^p(\mathcal I_{2h})} h^{\frac{p-r}{p}} \leq C h^{1-\frac{r}{p}}.
\end{equation}

Therefore, the desired compactness result is true. 
\end{proof}
}
\section{proof of Remark \ref{energy conservation rem}}
For $u=0$, taking the Fourier transform in $x$, we get for $\al\neq0$,
\begin{equation}\label{fourier in x}
\left\{\begin{array}{l}
\partial_t\wh{U}_1+i\al \wh{p}-i\al b\wh{H}_1-b'\wh{H}_2=0,\\
\partial_t\wh{U}_2+\partial_y\wh{p}-i\al b\wh{H}_2=0,\\
\partial_t\wh{H}_1+b'\wh{U}_2-i\al b\wh{U}_1=0,\\
\partial_t\wh{H}_2-i\al b\wh{U}_2=0,\\
i\al\wh{U}_1+\pa_y\wh{U}_2=0,\ \ i\al\wh{H}_1+\pa_y\wh{H}_2=0.
\end{array}\right.
\end{equation}
And we can diagonalize it to obtain
\begin{equation}\label{diag}
\left\{\begin{array}{l}
\partial_{tt}\wh{U}_2+\al^2\cA_{\al}\wh{U}_2=0,\\
\partial_{tt}(\wh{H}_2/b)+\al^2\cA_{\al}(\wh{H}_2/b)=0,
\end{array}\right.
\end{equation}
where $\cA_{\al}=(\pa_y^2-\al^2)^{-1}\big(b^2(\pa_y^2-\al^2)+2bb'\pa_y\big)$, and the details can be seen in \cite{RZ2017}. 
In order to prove the energy conservation law, we only need to prove the energy conservation on each frequency.

At first, we can show that 
\begin{align}\label{energy 1}
&\|\al^k\wh{U}_1\|^2_{L_y^2}+\|\al^k\wh{U}_2\|^2_{L_y^2}+\|\al^k\wh{H}_2\|^2_{L_y^2}+\|\al^k(\wh{H}_1-i(\al b)^{-1}b'\wh{H}_2)\|^2_{L_y^2}\nonumber\\
&=\|\al^k\wh{U}_{1,in}\|^2_{L_y^2}+\|\al^k\wh{U}_{2,in}\|^2_{L_y^2}+\|\al^k\wh{H}_{2,in}\|^2_{L_y^2}+\|\al^k(\wh{H}_{1,in}-i(\al b)^{-1}b'\wh{H}_{2,in})\|^2_{L_y^2}.
\end{align}
Indeed, taking $L^2$ inner product of $\eqref{diag}_2$ with $\overline{\al^{2k}\pa_t(\pa_y^2-\al^2)(\wh{H}_2/b)}$, integrating by part, taking the real part, we obtain
\begin{align*}
&\mathrm{Re}\int_{\bbT}\al^k\pa_{tt}(\wh{H}_2/b)\overline{\al^k\pa_t(\pa_y^2-\al^2)(\wh{H}_2/b)}\mathrm{d}y\\
&=-\mathrm{Re}\Big(\int_{\bbT}\al^k\pa_{tt}\pa_y(\wh{H}_2/b)\overline{\al^k\pa_t\pa_y(\wh{H}_2/b)}\mathrm{d}y+\al^{2+2k}\int_{\bbT}\pa_{tt}(\wh{H}_2/b)\overline{\pa_t(\wh{H}_2/b)}\mathrm{d}y\Big)\\
&=-\f12\f{\mathrm{d}}{\mathrm{d}t}\|\al^k\pa_t\pa_y(\wh{H}_2/b)\|_{L_y^2}^2-\f12\f{\mathrm{d}}{\mathrm{d}t}\|\al^{1+k}\pa_t(\wh{H}_2/b)\|_{L_y^2}^2
\end{align*}
and
\begin{align*}
&\mathrm{Re}\int_{\bbT}\cA_{\al}(\wh{H}_2/b)\overline{\al^{2k}\pa_t(\pa_y^2-\al^2)(\wh{H}_2/b)}\mathrm{d}y\\
&=\mathrm{Re}\int_{\bbT}(b^2(\pa_y^2-\al^2)+2bb'\pa_y)(\wh{H}_2/b)\overline{\al^{2k}\pa_t(\wh{H}_2/b)}\mathrm{d}y\\
&=-\mathrm{Re}\Big(\int_{\bbT}\al^kb^2\pa_y(\wh{H}_2/b)\overline{\al^k\pa_t\pa_y(\wh{H}_2/b)}\mathrm{d}y+\al^{2k+2}\int_{\bbT}\wh{H}_2\overline{\pa_t\wh{H}_2}\mathrm{d}y\Big)\\
&=-\f12\f{\mathrm{d}}{\mathrm{d}t}\|\al^k\pa_y(\wh{H}_2/b)\|_{L_y^2}^2-\f12\f{\mathrm{d}}{\mathrm{d}t}\|\al^{1+k}\wh{H}_2\|_{L_y^2}^2,
\end{align*}
and then
\beno
\f{\mathrm{d}}{\mathrm{d}t}\Big(\|\al^k\pa_t\pa_y(\wh{H}_2/b)\|_{L_y^2}^2+\|\al^{1+k}\pa_t(\wh{H}_2/b)\|_{L_y^2}^2+\|\al^{1+k} \pa_y(\wh{H}_2/b)\|_{L_y^2}^2+\|\al^{1+k}\wh{H}_2\|_{L_y^2}^2\Big)=0.
\eeno
By \eqref{fourier in x}, we get \eqref{energy 1} for $\al\neq0$.
Finally, we prove that for some constant $C>0$ independent of $t,\al$,
$$
C^{-1}\|\wh{H}_1\|_{L_y^2}\leq\|\wh{H}_1-i(\al b)^{-1}b'\wh{H}_2\|_{L_y^2}\leq C\|\wh{H}_1\|_{L_y^2}.
$$
By the condition that $\wh{H}_2(t,\al,y_1)=\wh{H}_2(t,\al,y_2)=0$, we have $\|\wh{H}_2\|_{L_y^2}\leq C\|\pa_y\wh{H}_2\|_{L_y^2}\leq C\|\al\wh{H}_1\|_{L_y^2}$. Thus,
\beno
\|\al^{-1+k}\pa_y(\wh{H}_2/b)\|_{L_y^2}\leq C\|\al^{-1+k}\pa_y\wh{H}_2\|_{L_y^2}+C\|\al^{-1+k}\wh{H}_2\|_{L_y^2}\leq C\|\al^k\wh{H}_1\|_{L_y^2}.
\eeno
And on the other hand, 
\begin{align*}
\|\al^k\wh{H}_1\|_{L_y^2}&\leq \|\al^k(\wh{H}_1-i(\al b)^{-1}b'\wh{H}_2)\|^2_{L_y^2}+\|\al^k(\al b)^{-1}b'\wh{H}_2\|^2_{L_y^2}\\
&\leq C\|\al^{-1+k}\pa_y(\wh{H}_2/b)\|_{L_y^2}\leq \|\al^k(\wh{H}_1-i(\al b)^{-1}b'\wh{H}_2)\|^2_{L_y^2},
\end{align*}
where we used the fact that $-i\al b\pa_y(\wh{H}_2/b)=\wh{H}_1-i(\al b)^{-1}b'\wh{H}_2$ in the last inequality. Since $C$ is independent of $t$ and $\al$, we obtain the proof of Remark \ref{energy conservation rem} by Plancherel identity.  

For the vorticity and current density, we have that by taking Fourier transform in $x$,
\begin{equation}\label{vorticity current}
\left\{\begin{array}{l}
\pa_t\wh{\om}=i\al b\wh{j}-b''\wh{H}_2,\\
\pa_t\wh{j}=i\al b\wh{\om}-2i\al b'\wh{U}_1+b''\wh{U}_2.
\end{array}\right.
\end{equation}
By taking $L^2$ inner product of $\eqref{vorticity current}_1$ with $\overline{\wh{w}}$ and $\eqref{vorticity current}_2$ with $\overline{\wh{j}}$, and taking the real part, we obtain that
\begin{align*}
\f12\f{\mathrm{d}}{\mathrm{d}t}\Big(\|\wh{\om}\|_{L_y^2}^2+\|\wh{j}\|_{L_y^2}^2\Big)&=\mathrm{Re}\Big(-\int_{\bbT}b''\wh{H}_2\overline{\wh{\om}}\mathrm{d}y-\int_{\mathbb{T}}2i\al b'\wh{U}_1\overline{\wh{j}}\mathrm{d}y+\int_{\bbT}b''\wh{U}_2\overline{\wh{j}}\mathrm{d}y\Big)\\
&\leq \|b''\wh{H}_2\|_{L_y^2}\|\wh{\om}\|_{L_y^2}+\|2\al b'\wh{U}_1\|_{L_y^2}\|\wh{j}\|_{L_y^2}+\|b''\wh{U}_2\|_{L_y^2}\|\wh{j}\|_{L_y^2}.
\end{align*}
And then, we get 
\begin{align*}
\f{\mathrm{d}}{\mathrm{d}t}\Big(\|\wh{\om}\|_{L_y^2}+\|\wh{j}\|_{L_y^2}\Big)\leq C\Big(\|\wh{H}_2\|_{L_y^2}+\|\al \wh{U}_1\|_{L_y^2}+\|\wh{U}_2\|_{L_y^2}\Big),
\end{align*}
by integrating in time, it holds
\begin{align*}
\|\wh{\om}\|_{L_y^2}+\|\wh{j}\|_{L_y^2}\leq \Big(\|\wh{\om}_0\|_{L_y^2}+\|\wh{j}_0\|_{L_y^2}\Big)+C\int_0^t\Big(\|\wh{H}_2\|_{L_y^2}+\|\al \wh{U}_1\|_{L_y^2}+\|\wh{U}_2\|_{L_y^2}\Big)\mathrm{d}\tau.
\end{align*}
Thus by using \eqref{energy 1}, we obtain the linear growth result of vorticity and current density.

Note that when considering the non-flowing plasma case, the system can be diagonalized to one with a self-adjoint operator $\cA_{\al}$. 

For general case with constant background velocity ($u=\text{const.}$), we have 
\begin{equation*}
\left\{\begin{array}{l}
\partial_t\wh{U}_1+i\al u\wh{U}_1+i\al \wh{p}-i\al b\wh{H}_1-b'\wh{H}_2=0,\\
\partial_t\wh{U}_2+i\al u\wh{U}_2+\partial_y\wh{p}-i\al b\wh{H}_2=0,\\
\partial_t\wh{H}_1+i\al u\wh{H}_1+b'\wh{U}_2-i\al b\wh{U}_1=0,\\
\partial_t\wh{H}_2+i\al u\wh{H}_2-i\al b\wh{U}_2=0,\\
i\al\wh{U}_1+\pa_y\wh{U}_2=0,\ \ i\al\wh{H}_1+\pa_y\wh{H}_2=0.
\end{array}\right.
\end{equation*}
Let us introduce $\wh{\tilde{U}}=e^{i\al ut}\wh{U}$, $\wh{\tilde{H}}=e^{i\al ut}\wh{H}$, $\wh{\tilde{p}}=e^{i\al ut}\wh{p}$ and $\wh{\tilde{\om}}=e^{i\al ut}\wh{\om}$, $\wh{\tilde{j}}=e^{i\al ut}\wh{j}$, then $(\wh{\tilde{U}},\wh{\tilde{H}},\wh{\tilde{p}})$ solves \eqref{fourier in x} and $(\wh{\tilde{\om}},\wh{\tilde{j}})$ solves \eqref{diag}. 

Thus Remark \ref{energy conservation rem} follows from the fact that
\beno
\left\|(\wh{\tilde{U}},\wh{\tilde{H}},\wh{\tilde{\om}},\wh{\tilde{j}})\right\|_{L^2_y}=\left\|(\wh{{U}},\wh{{H}},\wh{\om},\wh{j})\right\|_{L^2_y}.
\eeno

\end{CJK*}
\end{document}